\documentclass{proc-l}
\usepackage[all]{xy}
\usepackage{amssymb}
\usepackage{amsmath}
\usepackage{amsmath}

\usepackage[usenames]{color}

\def\rouge{ \textcolor{red} }

\swapnumbers
\newtheorem{thm}[subsection]{Theorem}

\newtheorem{lemma}[subsection]{Lemma}

\newcommand{\draftnote}[1]{}

\def\tstar {\,\,\tilde{*}\,\,}
\def\tsharp{\,\,\widetilde{\#}\,\,}

\def\d{\succ}
\def\g{\prec}

\def\YY{Y\!\!\!\! Y}

\def\t{\otimes}

\def\Sy{\mathbb{S}}

\def\YY{{\mathbb{Y}}}

\def\PP{{\mathcal{P}}}

\def\XXX{{\mathcal{X}}}

\def\ZZZ{{\mathcal{Z}}}

\def\Id{\mathrm{Id }}

\def\Im{\mathop{\rm Im }}

\def\Vect{\mathop{\rm Vect }}

\def\dim{\mathop{\rm dim }}

\def\alg{\textrm{-alg}}

\def\KK{\mathbb{K}}

\def\Fin{\mathtt{Fin}}
\def\Vect{\mathtt{Vect}}

\def\arbreA{\vcenter{\xymatrix@R=3pt@C=3pt{
&& \\
&*{}\ar@{-}[ur] \ar@{-}[ul] \ar@{-}[d]     &\\
&&
}}}

\newenvironment{proo}{\begin{trivlist} \item{\emph{Proof.}}}
  {\hfill $\square$ \end{trivlist}}

\begin{document}

\author[N. Bergeron]{Nantel Bergeron}
\address{Department of Mathematics and Statistics\\
2029 TEL Building\\
York University\\
North York, Ontario M3J 1P3, Canada}
\email{bergeron@yorku.ca }

\author[J.-L. Loday]{Jean-Louis Loday}
\address{Institut de Recherche Math\'ematique Avanc\'ee\\
    CNRS et Universit\'e de Strasbourg\\
    7 rue R. Descartes\\
    67084 Strasbourg Cedex, France}
\email{loday@math.unistra.fr}

\title{The symmetric operation in a free pre-Lie algebra is magmatic}


\begin{abstract} A pre-Lie product is a binary operation whose associator is symmetric in the last two variables. 
As a consequence its antisymmetrization is a Lie bracket. In this paper we study the symmetrization of the pre-Lie
 product. We show that it does not satisfy any other universal relation than commutativity. It means that the map
  from the free commutative-magmatic algebra to the free pre-Lie algebra induced by the symmetrization of the pre-Lie
   product is injective. This result is in contrast with the associative case, where the symmetrization gives rise to the notion of Jordan algebra.
   We first give a self-contained proof. Then we give a proof which uses the properties of dendriform and duplicial algebras.
\end{abstract}

\maketitle

\section*{Introduction} \label{S:int} 

A pre-Lie algebra is a vector space equipped with a binary operation $x*y$ whose associator is right-symmetric. Its name
 comes from the fact that the anti-symmetrization $[x,y]:= x*y-y*x$ of this binary operation is a Lie bracket. In this paper we
  investigate the symmetrization $x\#y:= x*y+y*x$ of the pre-Lie product. We show that, contrarily to the bracket, it does not
   satisfy any universal relation (but commutativity of course). In other words the map of operads
$$ComMag \to preLie$$
induced by the symmetrization is injective. Here $ComMag$ stands for the operad encoding the algebras equipped with a commutative
 binary operation (sometimes called commutative nonassociative algebras in the literature). We give two different proofs of this result. The first one is self-contained and rely on the combinatorics of trees. The second one uses the dendriform algebras. More precisley we prove that the map of operads 
 $$Mag\to Dend$$
 is injective by using the theory of generalized bialgebras for duplicial algebras. Vladimir Dotsenko informed us that he found an alternative proof by using Groebner basis \cite{Dotsenko}.

As a byproduct of our proof we show that the symmetrization of the binary operation generating the operad $NAP$ induces an injective morphism
$$ComMag \to NAP.$$
We recall that a $NAP$-algebra is equipped with a binary operation satisfying the quadratic relation $(xy)z=(xz)y$.

\section{Prerequisite on operads} Let $\Fin$ be the category of finite sets, and $\Vect$ be the category of vector spaces over a field $\KK$. 
Let $\underline{n}=\{1, \ldots , n\}$ be the standard finite set of cardinality $n$. We will use sometimes the fact that it is equipped with a total
 order. Any functor $ \Fin\to \Vect, S\mapsto {\mathcal O}[S]$, gives rise to an endofunctor of $\Vect$ by the formula
$${\mathcal O}(V) := \bigoplus_{n\geq 1}{\mathcal O}[\underline{n}]\t_{\Sy_{n}}V^{\t n}.$$
By definition an algebraic operad is such a functor equipped with a monoid structure, i.e.\ transformations of functors 
$$\gamma :{\mathcal O}\circ {\mathcal O} \to {\mathcal O}\quad \textrm{and} \quad \iota:\Id \to {\mathcal O}$$
such that $\gamma$ is associative and $\iota$ is a unit for $\gamma$. We observe that ${\mathcal O}(n):= {\mathcal O}[\underline{n}]$
 is a module over the symmetric group $\Sy_{n}$.

We recall that a type of algebras gives rise to an algebraic operad by taking the free algebra functor (cf.\ for instance \cite{MSS, LV}).

\section{Pre-Lie algebra} 

A \emph{pre-Lie algebra} is  a vector space $L$ together with a binary operation $*$ satisfying the
  relation
$$(x* y)* z-x*(y* z)=(x* z)* y-x*(z* y), \quad \forall x,y,z\in L.$$

The operad $\PP=preLie$ has been described in terms of trees by Chapoton and Livernet in \cite{ChapotonLivernet} as follows.
Given a finite set $S$ of cardinality $n$ let $\PP[S]$ be the  vector space spanned by the labelled rooted  trees on $n$
  vertices with distinct labels chosen in $S$. For example the space $\PP[\{a,b,c\}]$ is the linear span of the following trees:
   $$
 \raise -10pt
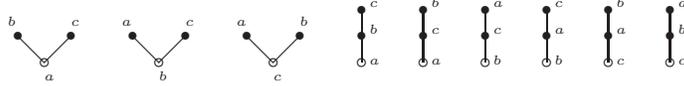
\hbox{ \begin{picture}(40,30)
       \put(10,6){\circle{3}}   \put(0,16){\circle*{3}}   \put(20,16){\circle*{3}}
       \put(10,6){\line(-1,1){10}}    \put(10,6){\line(1,1){10}}
       \put(10,-1){$\scriptscriptstyle a$}   \put(-4,20){$\scriptscriptstyle b$}   \put(20,20){$\scriptscriptstyle c$}  
      \end{picture}} 
 \raise -10pt\hbox{  \begin{picture}(40,30)
       \put(10,6){\circle{3}} \put(0,16){\circle*{3}} \put(20,16){\circle*{3}}
       \put(10,6){\line(-1,1){10}}  \put(10,6){\line(1,1){10}}
       \put(10,-1){$\scriptscriptstyle b$}   \put(-4,20){$\scriptscriptstyle a$} \put(20,20){$\scriptscriptstyle c$}  
      \end{picture}} 
 \raise -10pt\hbox{ \begin{picture}(40,30)
       \put(10,6){\circle{3}} \put(0,16){\circle*{3}}  \put(20,16){\circle*{3}}
       \put(10,6){\line(-1,1){10}}  \put(10,6){\line(1,1){10}}
       \put(10,-1){$\scriptscriptstyle c$}   \put(-4,20){$\scriptscriptstyle a$} \put(20,20){$\scriptscriptstyle b$}  
      \end{picture}} 
 \raise -10pt\hbox{ \begin{picture}(20,30)
       \put(0,6){\circle{3}} \put(0,16){\circle*{3}} \put(0,26){\circle*{3}}
       \put(0,6){\line(0,1){20}}
       \put(3,5){$\scriptscriptstyle a$} \put(3,17){$\scriptscriptstyle b$} \put(3,27){$\scriptscriptstyle c$}  
      \end{picture}} 
 \raise -10pt\hbox{ \begin{picture}(20,30)
       \put(0,6){\circle{3}}  \put(0,16){\circle*{3}} \put(0,26){\circle*{3}}
       \put(0,6){\line(0,1){20}}
       \put(3,5){$\scriptscriptstyle a$}  \put(3,17){$\scriptscriptstyle c$}  \put(3,27){$\scriptscriptstyle b$}  
      \end{picture}} 
 \raise -10pt\hbox{ \begin{picture}(20,30)
       \put(0,6){\circle{3}}  \put(0,16){\circle*{3}} \put(0,26){\circle*{3}}
       \put(0,6){\line(0,1){20}} 
        \put(3,5){$\scriptscriptstyle b$}  \put(3,17){$\scriptscriptstyle c$} \put(3,27){$\scriptscriptstyle a$}  
      \end{picture}} 
 \raise -10pt\hbox{ \begin{picture}(20,30)
       \put(0,6){\circle{3}}  \put(0,16){\circle*{3}} \put(0,26){\circle*{3}}
       \put(0,6){\line(0,1){20}}
       \put(3,5){$\scriptscriptstyle b$}  \put(3,17){$\scriptscriptstyle a$}  \put(3,27){$\scriptscriptstyle c$}  
      \end{picture}} 
 \raise -10pt\hbox{ \begin{picture}(20,30)
       \put(0,6){\circle{3}} \put(0,16){\circle*{3}}  \put(0,26){\circle*{3}}
       \put(0,6){\line(0,1){20}}
       \put(3,5){$\scriptscriptstyle c$}  \put(3,17){$\scriptscriptstyle a$}  \put(3,27){$\scriptscriptstyle b$}  
      \end{picture}} 
 \raise -10pt\hbox{ \begin{picture}(20,30)
       \put(0,6){\circle{3}}  \put(0,16){\circle*{3}} \put(0,26){\circle*{3}}
       \put(0,6){\line(0,1){20}}
       \put(3,5){$\scriptscriptstyle c$}   \put(3,17){$\scriptscriptstyle b$}   \put(3,27){$\scriptscriptstyle a$}  
      \end{picture}} 
 $$

The pre-Lie product is obtained by the following process.
Given two disjoint sets $I,J$ and two labelled trees $T\in\PP[I]$ and $Y\in\PP[J]$ one defines
$$T*Y:=
\def\objectstyle{\scriptstyle}
\def\labelstyle{\scriptstyle}
\sum_{t\in Vert(T)}
\vcenter{\xymatrix@-1.5pc{
*++[o][F-]{Y}\ar@{-*{\bullet}}[d]_>>{t}\\
*++[o][F-]{T}}}$$
where the sum is over all possible ways of {\sl grafting} the root of
the tree $Y$ on a vertex $t$ of $T$ by an edge. The root of the result is the root of $T$. For example
   $$
   \raise -10pt\hbox{ \begin{picture}(25,25)
       \put(10,6){\circle{3}}   \put(0,16){\circle*{3}}   \put(20,16){\circle*{3}}
       \put(10,6){\line(-1,1){10}}    \put(10,6){\line(1,1){10}}
       \put(10,-1){$\scriptscriptstyle c$}   \put(-4,20){$\scriptscriptstyle a$}  
        \put(20,20){$\scriptscriptstyle d$}  
       \end{picture}} \,*
    \raise -5pt\hbox{ \begin{picture}(10,5)
        \put(0,6){\circle{3}}     \put(3,5){$\scriptscriptstyle b$}    \end{picture}}
    \quad = \quad   
    \raise -10pt\hbox{ \begin{picture}(25,30)
       \put(10,6){\circle{3}}   \put(0,16){\circle*{3}}   \put(20,16){\circle*{3}}      \put(0,26){\circle*{3}}
       \put(10,6){\line(-1,1){10}}    \put(10,6){\line(1,1){10}}                           \put(0,16){\line(0,1){10}}        
       \put(10,-1){$\scriptscriptstyle c$}   \put(-5,18){$\scriptscriptstyle a$}  
        \put(20,20){$\scriptscriptstyle d$}                                           \put(-4,30){$\scriptscriptstyle b$} 
       \end{picture}}
    \quad + \quad   
    \raise -10pt\hbox{ \begin{picture}(25,30)
       \put(10,6){\circle{3}}   \put(0,16){\circle*{3}}   \put(20,16){\circle*{3}}      \put(10,16){\circle*{3}}
       \put(10,6){\line(-1,1){10}}    \put(10,6){\line(1,1){10}}                             \put(10,6){\line(0,1){10}}        
       \put(10,-1){$\scriptscriptstyle c$}   \put(-4,20){$\scriptscriptstyle a$}  
        \put(20,20){$\scriptscriptstyle d$}                                           \put(7,20){$\scriptscriptstyle b$} 
       \end{picture}}
     \quad + \quad   
     \raise -10pt\hbox{ \begin{picture}(25,30)
       \put(10,6){\circle{3}}   \put(0,16){\circle*{3}}   \put(20,16){\circle*{3}}      \put(20,26){\circle*{3}}
       \put(10,6){\line(-1,1){10}}    \put(10,6){\line(1,1){10}}                           \put(20,16){\line(0,1){10}}        
       \put(10,-1){$\scriptscriptstyle c$}   \put(-5,20){$\scriptscriptstyle a$}  
        \put(22,18){$\scriptscriptstyle d$}                                           \put(20,30){$\scriptscriptstyle b$} 
       \end{picture}}\,.
 $$

Any pre-Lie algebra $L$ gives rise to a Lie algebra whose bracket
is defined by the antisymmetric product
$$[x,y]:=x*y-y*x.$$
 This gives a morphism of operads 
$$Lie \longrightarrow preLie.$$
This morphism is injective and has been well studied in the literature, cf.\ for instance \cite{ChapotonLivernet} and \cite{Markl}.

In this paper we consider the symmetric product 
$$x\#y:=x*y+y*x.$$
 This new binary operation is related to the commutative-magmatic operad $ComMag$,
where a $ComMag$-algebra is  a vector space $A$ together with a binary operation $\cdot$ satisfying the relation
$$x\cdot y=y\cdot x, \quad \forall x,y\in A.$$
The symmetric product $\#$ in $preLie$ gives a morphism of operads 
$$\Phi\colon ComMag \longrightarrow preLie.$$
This has been much less studied and our main theorem is to show that this morphism is injective. Let us recall that $ComMag(n)$ is of dimension $(2n-3)!!= 1\times 3\times \cdots \times (2n-3)$ and that $preLie(n)$ is of dimension $n^{n-1}$.

An element in $ComMag[S]$ is a product $x\cdot y$ where $x\in ComMag[I], y\in ComMag[J]$ for some decomposition $S=I\cup J$ (disjoint union). Let us suppose that $S$ is equipped with a total order. Thanks to the commutativity of the product we can assume that the label $\max(S)$ is in $J$. In fact, there is a unique way to write an element in $ComMag[S]$ such that $\max(S)$ is on the right side and the two factors $x$ and $y$ satisfy also this property. We call it the normalized writing. For instance the elements of $ComMag(3)$ are written:
$$\quad (1\cdot 2)\cdot 3, \quad 1\cdot (2\cdot 3),\quad   2\cdot (1\cdot 3).$$

 We can now define $\Phi\colon ComMag\to \PP$ explicitly. If $S=\{a\}$ is of cardinality one, then $\Phi(a)=a$. If $|S|>1$, then any $Y\in ComMag[S]$ decomposes uniquely as $Y=V\cdot W$. We let $\Phi(Y):=\Phi(V)\#\Phi(W)$ and extend linearly to $ComMag[S]$.

\section{Symmetrizing the pre-Lie product}\label{mainsection}

The main result of this paper is the following

\begin{thm}\label{mainthm}
The morphim of operads  $\Phi\colon ComMag\to  preLie$ induced by the symmetric product $x\# y$  is injective.
\end{thm}

\begin{proof} {\bf Plan of the proof.} We introduce the auxiliary operation $x \tstar y$ on the labelled trees as follows. For any labelled trees $T$ and $Y$ (labelled on disjoint finite sets), $T  \tstar  Y$ is the unique tree obtained by connecting the root of $Y$ to the root of $T$. The root of $T \tstar Y$ is the former root of $T$:
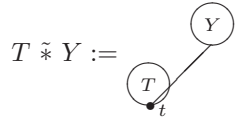
$$T\tstar  Y:=
\def\objectstyle{\scriptstyle}
\def\labelstyle{\scriptstyle}
\vcenter{\xymatrix@-1.5pc{
&*++[o][F-]{Y}\\
*++[o][F-]{T}&
}}  \hskip -40pt
 \raise -15pt\hbox{ \begin{picture}(60,35)
       \put(4,-2){\circle*{3}}   \put(4,-2){\line(1,1){23}}   \put(7,-6){$\scriptstyle t$}
      \end{picture}}
$$

 The proof of the theorem is divided into three steps.

-- step 1. We show that it is sufficient to prove injectivity for the symmetrization of the operation $\tstar$ (in place of $*$). 

-- step 2. We identify a subset of $\PP[S]$ which is in bijection with $ComMag[S]$ and we show that $\tstar$ realizes this bijection.

-- step 3. We show that the symmetrization of the operation $\tstar$ induces an injective map $ComMag[S] \to \PP[S]$.

\medskip

{\bf Step 1.} We need to show that the linear maps $\Phi_S\colon ComMag[S]\to \PP[S]$ are injective for all finite sets $S$.  We first introduce a filtration on  $\PP[S]$.
For a tree $T\in\PP[S]$, we say that $T$ is of {\sl degree} $d$ if the number of subtrees connected to the root of $T$ is $d$. We write $\deg(T)=d$.
This induces a filtration of $\PP[S]$. The pre-Lie product with respect to this filtration is
 $$ T * Y = T \tstar  Y + \hbox{ lower degree terms}$$
 where $T  \tstar  Y$ is the unique tree obtained by connecting the root of $Y$ to the root of $T$. So, for the purpose of this proof, we may use the product $\tstar $ instead of $*$. 
 
 \medskip
 
 {\bf Step 2.} We define a map $\Psi : ComMag[S] \to \PP[S]$, we  identify its image, and finally we prove that it is injective.
 
 Let $S$ be a totally ordered finite set. Let $t= x\cdot y$ be an element in $ComMag[S]$, whose writing has been normalized, that is $\mathrm{max}(S)$ is a label of $y$. We define inductively
 $$\Psi(t) := \Psi(x) \tstar \Psi(y),$$
 starting with $\Psi(1)= \circ_{1}$. So, we get
 $$\Psi(1\cdot 2)= 
  \raise -10pt\hbox{ \begin{picture}(40,20)
       \put(0,6){\circle{3}} \put(0,16){\circle*{3}}
       \put(0,6){\line(0,1){10}}
       \put(3,5){$\scriptscriptstyle 1$} \put(3,17){$\scriptscriptstyle 2$}
      \end{picture}},
      \Psi((1\cdot 2)\cdot 3)=      
 \raise -10pt\hbox{ \begin{picture}(40,30)
       \put(10,6){\circle{3}}   \put(0,16){\circle*{3}}   \put(20,16){\circle*{3}}
       \put(10,6){\line(-1,1){10}}    \put(10,6){\line(1,1){10}}
       \put(10,-1){$\scriptscriptstyle 1$}   \put(-4,20){$\scriptscriptstyle 2$}   
\put(20,20){$\scriptscriptstyle 3$}
      \end{picture}},
\Psi(1\cdot (2\cdot 3))=
 \raise -10pt\hbox{ \begin{picture}(20,30)
       \put(0,6){\circle{3}} \put(0,16){\circle*{3}} \put(0,26){\circle*{3}}
       \put(0,6){\line(0,1){20}}
       \put(3,5){$\scriptscriptstyle 1$} \put(3,17){$\scriptscriptstyle 2$} 
\put(3,27){$\scriptscriptstyle 3$}
      \end{picture}},
      \Psi(2\cdot (1\cdot 3))=
  \raise -10pt\hbox{ \begin{picture}(20,30)
       \put(0,6){\circle{3}}  \put(0,16){\circle*{3}} \put(0,26){\circle*{3}}
      \put(0,6){\line(0,1){10}}  \put(0,16){\line(0,1){10}}
       \put(3,5){$\scriptscriptstyle 2$}  \put(3,17){$\scriptscriptstyle 1$}  
\put(3,27){$\scriptscriptstyle 3$}
      \end{picture}}
      $$
   
    Any tree $T\in\PP[S]$ is uniquely  determined by its root $r$ and a set of subtrees $\{T_1,T_2,\ldots,T_\ell\}$ attached to the root. We denote this decomposition as $T=(r,\{T_1,T_2,\ldots,T_\ell\})$.  For 
$i\in\{1,2,\ldots,\ell\}$ let  $b_i=\max(T_i)$. We can assume that $b_1<b_2<\cdots 
<b_\ell$.
For a given set $S$, we denote by $A[S]$ the linear span of the trees in the image $\Psi[S]$ and by $B[S]$ the linear span of the trees of $\PP[S]$ not in $A[S]$. We have that $\PP[S]=A[S]\oplus B[S]$.
For any tree $T\in\PP[S]$, either the root is labelled with $\max(T)$ or there is a 
unique factorization
   \begin{equation}\label{eq:TX}
     T= X_1 \tstar  X_2
   \end{equation}
   where $\max(T)=b_\ell$ is a label in $X_2=T_\ell$.
For the factorization in Eq~(\ref{eq:TX})  to exist in $A[S]$ we must have that $r$ is not 
$\max(T)$. That is $r<b_\ell$.
We remark that recursively, $X_1$ will be factorizable if $r> b_{\ell-1}$. As we repeat 
this we see that 
$T\in B[S]$ if $r>b_1$. In particular, we get a characterization of the $T\in A[S]$ as 
follows.
    
{\bf Claim 1.} The image $\Psi(x\cdot y)$ is a tree $T=(r,\{T_1,T_2,\ldots,T_\ell\})$, which is characterized by the following properties:

{\rm (a)} $T_i\in A[S_i]$ for all $i\in\{1,2,\ldots,\ell\}$

{\rm (b)} $r<b_1$.

{\it Proof of claim 1}. The tree $T_{i}$ either comes from $\Psi(x)$ or is $\Psi(y)$. In both cases it is in $A[S_{i}]$ by induction. The index $b_{i}$ either come from  $\Psi(x)$ or is $\max \Psi(y)= \max(y)$. In both cases we have $r<b_{i}$ as expected.

{\bf Claim 2.} $\Psi$ is a bijection onto its image $A[S]$.

{\it Proof of claim 2}. Starting with $T\in \Im \Psi$, by claim 1 we have $T=T_{1} \tstar T_{2}$ where $\max(T)\in T_{2}$ and $T_{1}=\Psi(x_{1}), T_{2}=\Psi(x_{2})$. So we have $T=\Psi(x_{1}\cdot x_{2})$. So we have constructed $\Psi^{-1}$. This shows that $\Psi$ is an isomorphism from $ComMag[S]$ to $A[S]$.

\medskip

 {\bf Step 3.} 
 For $T\in A[S]$ and its factorisation $T=X_1 \tstar  X_2$, for which $\max(T)=\max(T_{2})$, we define $d(T)=|X_2|$, the size of $X_2$ (the number of vertices).

 $$t=\qquad 
 \raise -10pt\hbox{ \begin{picture}(40,30)
       \put(10,6){\circle{3}}   \put(0,16){\circle*{3}}   \put(20,16){\circle*{3}}
       \put(10,6){\line(-1,1){10}}    \put(10,6){\line(1,1){10}}
       \put(10,-1){$\scriptscriptstyle 1$}   \put(-4,20){$\scriptscriptstyle 2$}   
\put(20,20){$\scriptscriptstyle 3$}
      \end{picture}}\qquad
 \raise -10pt\hbox{ \begin{picture}(20,30)
       \put(0,6){\circle{3}} \put(0,16){\circle*{3}} \put(0,26){\circle*{3}}
       \put(0,6){\line(0,1){20}}
       \put(3,5){$\scriptscriptstyle 1$} \put(3,17){$\scriptscriptstyle 2$} 
\put(3,27){$\scriptscriptstyle 3$}
      \end{picture}}\qquad
  \raise -10pt\hbox{ \begin{picture}(20,30)
       \put(0,6){\circle{3}}  \put(0,16){\circle*{3}} \put(0,26){\circle*{3}}
      \put(0,6){\line(0,1){10}}  \put(0,16){\line(0,1){10}}
       \put(3,5){$\scriptscriptstyle 2$}  \put(3,17){$\scriptscriptstyle 1$}  
\put(3,27){$\scriptscriptstyle 3$}
      \end{picture}}
      $$
      $$ d(t)=  \hskip 1cm  1 \hskip 1.7cm 2  \hskip 1.5cm 2  \hskip 1.1cm $$
We can now turn to the map $\widetilde{\Phi}:ComMag[S] \to \PP[S]$ given by 
$$\widetilde{\Phi}(x\cdot y) = \Phi(x) \tsharp \Phi(y):= \Phi(x) \tstar \Phi(y)+ \Phi(y) \tstar \Phi(x).$$
Our purpose is to show that it is injective. 
Given a tree $Y\in ComMag[S]$ the tree $\Psi(Y)\in\PP[S]$ is clearly in the support of  $\widetilde{\Phi}(Y)$. 
We want to show that the image of $\widetilde{\Phi}[S]$ is linearly isomorphic to $A[S]$. This would show the injectivity of our map. The principle of the proof is to show that in $\widetilde{\Phi}(T)$ it may appear several elements of $A[S]$ of various $d$ values, but there is only one of maximal $d$ value. 
For instance $\widetilde{\Phi}(1\cdot (2\cdot 3))$ contain only one element $T_1 \in A[S]$ where  $d(T_1)=2$, $\widetilde{\Phi}((1\cdot 2)\cdot 3)$ contain only one element $T_2 \in A[S]$ where  $d(T_2)=1$, while $\widetilde{\Phi}(2\cdot (1\cdot 3))$ contains two elements $T_2, T_3 \in A[S]$ where $d(T_3)=2$:
$$(1\cdot (2\cdot 3))\mapsto \qquad 
 \raise -10pt\hbox{ \begin{picture}(20,30)
       \put(0,6){\circle{3}}  \put(0,16){\circle*{3}} \put(0,26){\circle*{3}}
       \put(0,6){\line(0,1){20}}
       \put(3,5){$\scriptscriptstyle 1$}   \put(3,17){$\scriptscriptstyle 2$}   \put(3,27){$\scriptscriptstyle 3$}  
      \end{picture}} +\
       \raise -10pt\hbox{ \begin{picture}(20,30)
       \put(0,6){\circle{3}}  \put(0,16){\circle*{3}} \put(0,26){\circle*{3}}
       \put(0,6){\line(0,1){20}}
       \put(3,5){$\scriptscriptstyle 1$}   \put(3,17){$\scriptscriptstyle 3$}   \put(3,27){$\scriptscriptstyle 2$}  
             \end{picture}} +\ 
 \raise -10pt\hbox{ \begin{picture}(40,30)
       \put(10,6){\circle{3}}   \put(0,16){\circle*{3}}   \put(20,16){\circle*{3}}
       \put(10,6){\line(-1,1){10}}    \put(10,6){\line(1,1){10}}
       \put(10,-1){$\scriptscriptstyle 2$}   \put(-4,20){$\scriptscriptstyle 1$}   \put(20,20){$\scriptscriptstyle 3$}  
      \end{picture}} +\ 
       \raise -10pt\hbox{ \begin{picture}(40,30)
       \put(10,6){\circle{3}}   \put(0,16){\circle*{3}}   \put(20,16){\circle*{3}}
       \put(10,6){\line(-1,1){10}}    \put(10,6){\line(1,1){10}}
       \put(10,-1){$\scriptscriptstyle 3$}   \put(-4,20){$\scriptscriptstyle 1$}   \put(20,20){$\scriptscriptstyle 2$}  
      \end{picture}} 
 $$

$$((1\cdot 2)\cdot 3) \mapsto\qquad 
   \raise -10pt\hbox{ \begin{picture}(40,30)
       \put(10,6){\circle{3}}   \put(0,16){\circle*{3}}   \put(20,16){\circle*{3}}
       \put(10,6){\line(-1,1){10}}    \put(10,6){\line(1,1){10}}
       \put(10,-1){$\scriptscriptstyle 1$}   \put(-4,20){$\scriptscriptstyle 2$}   \put(20,20){$\scriptscriptstyle 3$}  
      \end{picture}} +\ 
       \raise -10pt\hbox{ \begin{picture}(40,30)
       \put(10,6){\circle{3}}   \put(0,16){\circle*{3}}   \put(20,16){\circle*{3}}
       \put(10,6){\line(-1,1){10}}    \put(10,6){\line(1,1){10}}
       \put(10,-1){$\scriptscriptstyle 2$}   \put(-4,20){$\scriptscriptstyle 1$}   \put(20,20){$\scriptscriptstyle 3$}  
      \end{picture}} +\ 
       \raise -10pt\hbox{ \begin{picture}(20,30)
       \put(0,6){\circle{3}}  \put(0,16){\circle*{3}} \put(0,26){\circle*{3}}
       \put(0,6){\line(0,1){20}}
       \put(3,5){$\scriptscriptstyle 3$}   \put(3,17){$\scriptscriptstyle 1$}   \put(3,27){$\scriptscriptstyle 2$}  
      \end{picture}} +\ 
       \raise -10pt\hbox{ \begin{picture}(20,30)
       \put(0,6){\circle{3}}  \put(0,16){\circle*{3}} \put(0,26){\circle*{3}}
       \put(0,6){\line(0,1){20}}
       \put(3,5){$\scriptscriptstyle 3$}   \put(3,17){$\scriptscriptstyle 2$}   \put(3,27){$\scriptscriptstyle 1$}  
             \end{picture}}
 $$

$$(2\cdot (1\cdot 3)) \mapsto\qquad 
  \raise -10pt\hbox{ \begin{picture}(20,30)
       \put(0,6){\circle{3}}  \put(0,16){\circle*{3}} \put(0,26){\circle*{3}}
       \put(0,6){\line(0,1){20}}
       \put(3,5){$\scriptscriptstyle 2$}   \put(3,17){$\scriptscriptstyle 1$}   \put(3,27){$\scriptscriptstyle 3$}  
      \end{picture}} +\ 
       \raise -10pt\hbox{ \begin{picture}(20,30)
       \put(0,6){\circle{3}}  \put(0,16){\circle*{3}} \put(0,26){\circle*{3}}
       \put(0,6){\line(0,1){20}}
       \put(3,5){$\scriptscriptstyle 2$}   \put(3,17){$\scriptscriptstyle 3$}   \put(3,27){$\scriptscriptstyle 1$}  
             \end{picture}} +\ 
   \raise -10pt\hbox{ \begin{picture}(40,30)
       \put(10,6){\circle{3}}   \put(0,16){\circle*{3}}   \put(20,16){\circle*{3}}
       \put(10,6){\line(-1,1){10}}    \put(10,6){\line(1,1){10}}
       \put(10,-1){$\scriptscriptstyle 1$}   \put(-4,20){$\scriptscriptstyle 2$}   \put(20,20){$\scriptscriptstyle 3$}  
      \end{picture}} +\ 
       \raise -10pt\hbox{ \begin{picture}(40,30)
       \put(10,6){\circle{3}}   \put(0,16){\circle*{3}}   \put(20,16){\circle*{3}}
       \put(10,6){\line(-1,1){10}}    \put(10,6){\line(1,1){10}}
       \put(10,-1){$\scriptscriptstyle 3$}   \put(-4,20){$\scriptscriptstyle 1$}   \put(20,20){$\scriptscriptstyle 2$}  
      \end{picture}}
 $$

We proceed by induction on $|S|$. 
If $|S|=1$, then $\widetilde{\Phi}$ is a linear isomorphism and the result is clear. Assume $|S|>1$.
Our induction hypothesis is that the image of $\widetilde{\Phi}[S']$ is isomorphic to $A[S']$ for all $|S'|<|S|$.
More precisely, we assume that the image of $\widetilde{\Phi}[S']$ is exactly $A[S']$ modulo terms in $B[S']$.
This is certainly true for $|S|=1$.

To obtain this, we order the trees of $A[S]$ (a basis) as follows. 
Let $T,T'$  be trees in $A[S]$. We say that $T<T'$ whenever $d(T)>d(T')$.
When $|S|>1$ we have that 
  $$ComMag[S] = \bigoplus_{S=S_1 + S_2 \atop \min(S_1)<\min(S_2)} ComMag[S_1]\cdot ComMag[S_2],$$
  where $S=S_1+S_2$ denotes a set partition of $S$ into two non-empty disjoint parts. 
  This implies that 
    $$ \widetilde{\Phi}[S] =  \bigoplus_{S=S_1 + S_2 \atop \min(S_1)<\min(S_2)} 
    \widetilde{\Phi}[S_1] \tsharp \widetilde{\Phi}[S_1],$$
where 
$V\tsharp W= V\tstar  W+V\tstar  W$. For $i=1$ or $2$, by induction hypothesis, we may assume that a basis of the image of $\widetilde{\Phi}[S_i]$ 
is of the form $\{X_i+W_i: X_i \in A[S_i]\}$ where  the $X_i$'s are single trees and $W_i \in B[S_i]$.  Let $T$ be any tree in $A[S]$. 
Using the decomposition in Eq~(\ref{eq:TX}), we have that there are unique $S=S_1+S_2$ and trees $X_1\in A[S_1]$ and $X_2\in A[S_2]$ such
that $T=X_1 \tstar  X_2$. By induction hypothesis, there are basis elements $X_1+W_1$ in the image of $\widetilde{\Phi}[S_1]$ and $X_2+W_2$ in the image of $\widetilde{\Phi}[S_2]$. By construction we have that 
\begin{equation} \label{eq:XW}
 (X_1+W_1)\tsharp  (X_2+W_2) = X_1 \tstar  X_2 + X_2 \tstar  X_1 + \cdots
\end{equation}
We study this equation. In the following let $b=\max (T)$. Remark that $b$ is necessarily a label in $X_2$

Let $T'=X_2 \tstar  X_1$.
If $T'\in A[S]$, then $b$ must be in a proper subtree $V_2$ of $T'$ and the decomposition as in Eq~(\ref{eq:TX}) for $T'$ is $T'=V_1\tstar V_2$. We remark that $V_2$ must be a proper subtree of $X_2$ (since the root of $T'$ is the root of $X_2$), hence $|X_2|>|V_2|$. This implies that $T<T'$.
It may happen that $T'\in B[S]$ but this does not cause any problem. 
We need to inspect the other terms in Eq~(\ref{eq:XW}). Let $Z_i$ be a tree in the support of $W_i$, in particular  $Z_i \in B[S_i]$.
For any tree $Y$, the tree $Z_i$ will eventually be a factor as in Eq~(\ref{eq:TX}) for the tree $Y\tstar Z_i$ (this may take more than one recursive step). This shows that  $(X_1+W_1) \tstar  W_2$ and  $(X_2+W_2) \tstar  W_1$ are both in $B[S]$. 
For elements of the form $Z_i\tstar Y$ we have to be more careful. 
For $Z=Z_1 \tstar  X_2$,  we have that $\max (Z)=b$ is in $X_2$, this implies that $Z=Z_1 \tstar  X_2$
is the unique decomposition of $Z$ as in Eq~(\ref{eq:TX}) so $Z\not\in A[S]$. This shows that $W_1\tstar X_2\in B[S]$.
Finally, for $Z=Z_2 \tstar  X_1$, if $Z\in A[S]$, then $b$ must be in a proper subtree $V_2$ of $Z$ and the decomposition as in Eq~(\ref{eq:TX}) for $Z$ is $Z=V_1\tstar V_2$. We remark that $V_2$ must be a proper subtree of $Z_2$, hence $|X_2|=|Z_2|>|V_2|$. This implies that $T<Z$. 
For all $T\in A[S]$, we have shown that there is an element in the image of $\widetilde{\Phi}[S]$ given by Eq~(\ref{eq:XW}) which takes the form
\begin{equation}
 (X_1+W_1)\tsharp  (X_2+W_2) = T + \sum_{T'\in A[S] \atop T<T'} c_{T'} T' + W
\end{equation}
where $W\in B[S]$.

We have produced elements in the image of $\widetilde{\Phi}[S]$ that are linearly independent (different leading terms), one for each tree $T\in A[S]$. Since $\dim(A[S])=\dim( ComMag[S])$ we have that $\widetilde{\Phi}[S]$ is injective.
\end{proof}

\bigskip


\section{ Comparison with associative, Jordan and NAP-algebras}

\subsection{$NAP$-algebras} The operation $\tstar$ used as an auxiliary operation in the proof of the main theorem satisfies the following relation:
$$(x \tstar y) \tstar z = (x \tstar z) \tstar y $$
for any labelled trees $x,y,z$. Algebras with one binary operation satisfying this identity already occured in the literature in the work of M.\ Livernet \cite{Livernet}, and were called $NAP$-algebras (for NonAssociative Perm). She showed that the labelled trees do describe the associated operad and that the product is precisely given by the operation $\tstar$ on trees described in the proof of Theorem \ref{mainthm}. Therefore, as a byproduct of the proof of the main theorem we obtain the following:

\begin{thm}\label{thmNAP} The symmetrization of the operation $\tstar$ induces an injection of operads
$$ComMag \to NAP.$$
\end{thm}

\subsection{Jordan algebras} Let  $x\cdot y :=xy + yx$ where $xy$ is an associative operation. It is easy to check that the 3 operations $(x\cdot y)\cdot z,\ (x\cdot z)\cdot y,\ (y\cdot z)\cdot x$ are linearly independent in the free associative algebra. However the 15 operations in arity 4 are not linearly independent: they satisfy the formula:
$$
(x\cdot y)\cdot (t\cdot z) + (x\cdot z)\cdot (t\cdot y)+(y\cdot z)\cdot (t\cdot x)=((x\cdot y)\cdot t)\cdot z+ ((x\cdot z)\cdot t)\cdot y+ ((y\cdot z)\cdot t)\cdot x.
$$
Since this formula is invariant under any permutation of the three variables $x,y,z$ it is equivalent, when 3 is invertible, to a formula in two variables only. It leads to the definition of a Jordan algebra which is a vector space equipped with a symmetric binary operation $a\cdot b$ satisfying the relation:
$$(a^{\cdot 2})\cdot (b\cdot a) = ((a^{\cdot 2})\cdot b)\cdot a\ .$$

Let us denote by $Jord$ the operad governing Jordan algebras. Since any symmetrized associative product satisfies the Jordan relation there is a morphism of operads
$$Jord \to As\ .$$
Contrarily to the Lie-As case and the ComMag-preLie case, this morphism is not injective. In other words there are relations satisfied by the symmetrized associative product which are not consequences of the Jordan relation. The first examples appear in arity 8: the Glennie relations, see p.\ 79 of \cite{ZSSS}.

\section{Conjecture on the factorization of $preLie$}\label{factorization}

\subsection{On the injectivity of $Lie\to As$} There are several ways of proving the injectivity of the morphism of operads  $Lie\to As$. One of them is to take advantage of the existence of the notion of ``cocommutative bialgebra''. Indeed, the Poincar\'e-Birkhoff-Witt theorem implies that there is an isomorphism of endofunctors of $Vect$
$$As = Com \circ Lie.$$

Therefore it is natural to look for a similar statement when replacing Lie-As by ComMag-preLie.

\subsection{Conjecture: splitting of $preLie$}  There exists an $\Sy$-module $\XXX$ such that 
$$preLie= \XXX \circ ComMag.$$
{}From this formula the dimension of the spaces $\XXX(n)$ would be 

$$1, 1, 3, 16, 120, 1146, 13258, \ldots, $$

\subsection{Black and red trees} Since the conjecture is only about the structure of $\Sy$-module of $preLie$ and since $NAP$ has the same underlying $\Sy$-module we could as well put $NAP$ in the statement. For the rest of the section we denote by $\PP$ the underlying $\Sy$-module of both $preLie$ and $NAP$.

Here are some arguments in favor of the conjecture. In this section we construct  a linear isomorphism
\begin{equation}\label{redblack}
\PP  [S] \longrightarrow  \bigoplus_{\phi\vdash S} \XXX[\phi] \times \prod_{R\in\phi} ComMag[R].
\end{equation}
where $\phi\vdash S$ denotes a set partition of $S$. Unfortunately, the family of spaces $\XXX$ that we construct is not defining an $\Sy$-module. 
We believe that our construction can be modified to realize this, but we have not yet found the right way to do so.

In Step~2 of Section~\ref{mainsection} we constructed an injection $\Psi\colon ComMag\to\PP$.
We say that a tree of $\PP$ that is in  the image of $\Psi$ is a tree of type $A$. In Claim 1 of Step 2 we have a characterization of the tree of type $A$.
To obtain the linear isomorphism in~(\ref{redblack}), we decompose 
(uniquely) any tree of $\PP $ into special trees (of type $\XXX$) labelled with trees of type $A$. 

To get this decomposition, we work recursively from the leaves down to the root.  We will mark
some edges red (double lines) and others will remain black. The black connected components will represent
the trees of type $A$ and the double red edges will connect them to form a tree.
We also have that a double red edge is always between the root of a black tree to the maximum 
in the black tree below connected to it by that double red edge.

If $T$ is a single leaf, then there are no edges to colour.
Now  assume we have  $T=(r,\{T_1,T_2,\ldots,T_\ell\})$ and all the trees $T_i$ for
$i\in\{1,2,\ldots,\ell\}$ have been colored. Let $T'_i$ be the black connected component
at the root of $T_i$. Let $b'_i=\max(T'_i)$ and assume $b'_1<b'_2<\cdots <b'_\ell$.
We assume (by construction up to that point) that
in $T_i$, all the double red edges connected to $T'_i$ are connected to $b'_i$.
If  $r>b'_1$ or if $r<b'_1$ and $T'_i\ne T_i$ for $1\le i< \ell$, then we mark double red all 
edges out of the root $r$  otherwise, we mark them black. We proceed to the next step.
This will mark the edges of $T$ double red or black. If we cut all marked edges, thanks to
Claim 1 we get a forest of trees of type $A$. They are connected by
marked edges and this structure forms a tree. This decomposition is unique.
Remark that the condition ensures that the bicolored tree $T$ that we obtain is black at $r$ 
only if all double red edges in the black
connected component of $r$ are connected to the maximum vertex of that component. Otherwise, 
the tree is double red at $r$.

The tree below show how the special double red marking works. Let
$$
T=\quad  \raise -10pt
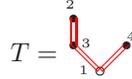
\hbox{ \begin{picture}(30,30)
       \put(0,6){\circle{3}} \put(-10,16){\circle*{3}} \put(10,16){\circle*{3}}
       \put(-10,26){\circle*{3}}
     \rouge{   \put(-1,6){\line(-1,1){10}}  \put(1,6){\line(-1,1){10}}  
                    \put(-1,6){\line(1,1){10}}   \put(1,6){\line(1,1){10}}
                    \put(-9,16){\line(0,1){10}}  \put(-11,16){\line(0,1){10}} }
       \put(-8,5){$\scriptscriptstyle 1$} \put(-7,15){$\scriptscriptstyle 3$}
       \put(10,18){$\scriptscriptstyle 4$}  \put(-13,30){$\scriptscriptstyle 2$}
       \end{picture}}
$$
At the root $1$ of $T$, we have that $1<3$ but the tree containing $3$ is not all black, 
so we are marking the edges red.
If we look at the tree
$$
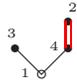
\raise -10pt\hbox{ \begin{picture}(30,30)
       \put(0,6){\circle{3}} \put(-10,16){\circle*{3}} \put(10,16){\circle*{3}}
       \put(10,26){\circle*{3}}
       \put(0,6){\line(-1,1){10}}
       \put(0,6){\line(1,1){10}}
       \put(-8,5){$\scriptscriptstyle 1$} \put(-13,20){$\scriptscriptstyle 3$}
       \put(3,15){$\scriptscriptstyle 4$}  \put(10,30){$\scriptscriptstyle 2$}
    \rouge{       \put(9,16){\line(0,1){10}}   \put(11,16){\line(0,1){10}} }
      \end{picture}}
$$
at $1$, then we see that the subtree $3$ is all black and the only double red edges are connected to $4$ the maximum, 
so the edges out of $1$ are black.

For example we give a list of the decomposition for all trees of $\PP [\{1,2, 3,4\}]$ :
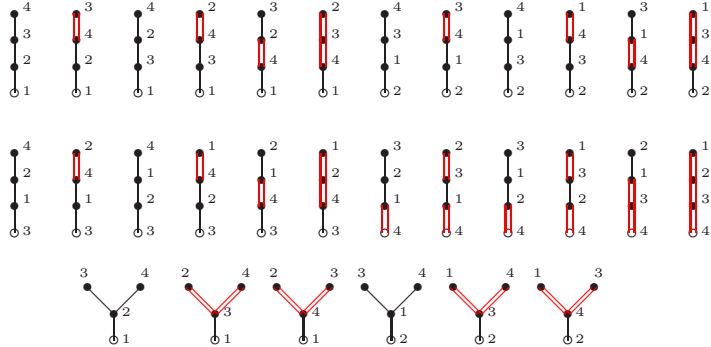
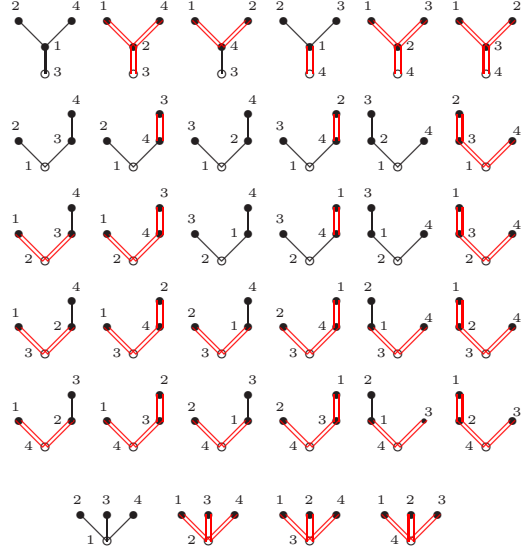
  $$
 \raise -10pt\hbox{ \begin{picture}(20,40)
       \put(0,6){\circle{3}} \put(0,16){\circle*{3}} \put(0,26){\circle*{3}} 
       \put(0,36){\circle*{3}}
       \put(0,6){\line(0,1){10}}
       \put(0,16){\line(0,1){10}}
       \put(0,26){\line(0,1){10}}
       \put(3,5){$\scriptscriptstyle 1$} \put(3,17){$\scriptscriptstyle 2$} 
       \put(3,27){$\scriptscriptstyle 3$}  \put(3,37){$\scriptscriptstyle 4$}
       \end{picture}}
 \raise -10pt\hbox{ \begin{picture}(20,40)
       \put(0,6){\circle{3}} \put(0,16){\circle*{3}} \put(0,26){\circle*{3}} 
       \put(0,36){\circle*{3}}
       \put(0,6){\line(0,1){10}}
       \put(0,16){\line(0,1){10}}
       \put(3,5){$\scriptscriptstyle 1$} \put(3,17){$\scriptscriptstyle 2$} 
       \put(3,27){$\scriptscriptstyle 4$}  \put(3,37){$\scriptscriptstyle 3$}
    \rouge{\put(-1,26){\line(0,1){10}} \put(1,26){\line(0,1){10}}}
       \end{picture}}
 \raise -10pt\hbox{ \begin{picture}(20,40)
       \put(0,6){\circle{3}} \put(0,16){\circle*{3}} \put(0,26){\circle*{3}} 
       \put(0,36){\circle*{3}}
       \put(0,6){\line(0,1){10}}
       \put(0,16){\line(0,1){10}}
       \put(0,26){\line(0,1){10}}
       \put(3,5){$\scriptscriptstyle 1$} \put(3,17){$\scriptscriptstyle 3$} 
      \put(3,27){$\scriptscriptstyle 2$}  \put(3,37){$\scriptscriptstyle 4$}
      \end{picture}}
 \raise -10pt\hbox{ \begin{picture}(20,40)
       \put(0,6){\circle{3}} \put(0,16){\circle*{3}} \put(0,26){\circle*{3}} 
       \put(0,36){\circle*{3}}
       \put(0,6){\line(0,1){10}}
       \put(0,16){\line(0,1){10}}
       \put(3,5){$\scriptscriptstyle 1$} \put(3,17){$\scriptscriptstyle 3$} 
       \put(3,27){$\scriptscriptstyle 4$}  \put(3,37){$\scriptscriptstyle 2$}
    \rouge{\put(-1,26){\line(0,1){10}} \put(1,26){\line(0,1){10}}}
       \end{picture}}
 \raise -10pt\hbox{ \begin{picture}(20,40)
       \put(0,6){\circle{3}} \put(0,16){\circle*{3}} \put(0,26){\circle*{3}} 
       \put(0,36){\circle*{3}}
       \put(0,6){\line(0,1){10}}
       \put(0,26){\line(0,1){10}}
       \put(3,5){$\scriptscriptstyle 1$} \put(3,17){$\scriptscriptstyle 4$} 
       \put(3,27){$\scriptscriptstyle 2$}  \put(3,37){$\scriptscriptstyle 3$}
    \rouge{   \put(-1,16){\line(0,1){10}}   \put(1,16){\line(0,1){10}} }
       \end{picture}}
 \raise -10pt\hbox{ \begin{picture}(20,40)
       \put(0,6){\circle{3}} \put(0,16){\circle*{3}} \put(0,26){\circle*{3}} 
       \put(0,36){\circle*{3}}
       \put(0,6){\line(0,1){10}}
       \put(3,5){$\scriptscriptstyle 1$} \put(3,17){$\scriptscriptstyle 4$} 
       \put(3,27){$\scriptscriptstyle 3$}  \put(3,37){$\scriptscriptstyle 2$}
    \rouge{\put(-1,16){\line(0,1){10}} \put(1,16){\line(0,1){10}}  \put(-1,26){\line(0,1){10}}   \put(1,26){\line(0,1){10}}}
       \end{picture}}
 \raise -10pt\hbox{ \begin{picture}(20,40)
       \put(0,6){\circle{3}} \put(0,16){\circle*{3}} \put(0,26){\circle*{3}} 
       \put(0,36){\circle*{3}}
       \put(0,6){\line(0,1){10}}
       \put(0,16){\line(0,1){10}}
       \put(0,26){\line(0,1){10}}
       \put(3,5){$\scriptscriptstyle 2$} \put(3,17){$\scriptscriptstyle 1$} 
       \put(3,27){$\scriptscriptstyle 3$}  \put(3,37){$\scriptscriptstyle 4$}
       \end{picture}}
 \raise -10pt\hbox{ \begin{picture}(20,40)
       \put(0,6){\circle{3}} \put(0,16){\circle*{3}} \put(0,26){\circle*{3}} 
       \put(0,36){\circle*{3}}
       \put(0,6){\line(0,1){10}}
       \put(0,16){\line(0,1){10}}
       \put(3,5){$\scriptscriptstyle 2$} \put(3,17){$\scriptscriptstyle 1$} 
       \put(3,27){$\scriptscriptstyle 4$}  \put(3,37){$\scriptscriptstyle 3$}
    \rouge{\put(-1,26){\line(0,1){10}} \put(1,26){\line(0,1){10}} }
       \end{picture}}
 \raise -10pt\hbox{ \begin{picture}(20,40)
       \put(0,6){\circle{3}} \put(0,16){\circle*{3}} \put(0,26){\circle*{3}} 
       \put(0,36){\circle*{3}}
       \put(0,6){\line(0,1){10}}
       \put(0,16){\line(0,1){10}}
       \put(0,26){\line(0,1){10}}
       \put(3,5){$\scriptscriptstyle 2$} \put(3,17){$\scriptscriptstyle 3$} 
       \put(3,27){$\scriptscriptstyle 1$}  \put(3,37){$\scriptscriptstyle 4$}
      \end{picture}}
 \raise -10pt\hbox{ \begin{picture}(20,40)
       \put(0,6){\circle{3}} \put(0,16){\circle*{3}} \put(0,26){\circle*{3}} 
       \put(0,36){\circle*{3}}
       \put(0,6){\line(0,1){10}}
       \put(0,16){\line(0,1){10}}
       \put(3,5){$\scriptscriptstyle 2$} \put(3,17){$\scriptscriptstyle 3$} 
       \put(3,27){$\scriptscriptstyle 4$}  \put(3,37){$\scriptscriptstyle 1$}
    \rouge{\put(-1,26){\line(0,1){10}} \put(1,26){\line(0,1){10}} }
       \end{picture}}
 \raise -10pt\hbox{ \begin{picture}(20,40)
       \put(0,6){\circle{3}} \put(0,16){\circle*{3}} \put(0,26){\circle*{3}} 
       \put(0,36){\circle*{3}}
       \put(0,6){\line(0,1){10}}
       \put(0,26){\line(0,1){10}}
       \put(3,5){$\scriptscriptstyle 2$} \put(3,17){$\scriptscriptstyle 4$} 
       \put(3,27){$\scriptscriptstyle 1$}  \put(3,37){$\scriptscriptstyle 3$}
    \rouge{   \put(-1,16){\line(0,1){10}}    \put(1,16){\line(0,1){10}} }
       \end{picture}}
 \raise -10pt\hbox{ \begin{picture}(20,40)
       \put(0,6){\circle{3}} \put(0,16){\circle*{3}} \put(0,26){\circle*{3}} 
       \put(0,36){\circle*{3}}
    \rouge{  \put(-1,16){\line(0,1){10}}   \put(1,16){\line(0,1){10}}   \put(-1,26){\line(0,1){10}}   \put(1,26){\line(0,1){10}}}
       \put(0,6){\line(0,1){10}}
       \put(3,5){$\scriptscriptstyle 2$} \put(3,17){$\scriptscriptstyle 4$} 
       \put(3,27){$\scriptscriptstyle 3$}  \put(3,37){$\scriptscriptstyle 1$}
       \end{picture}}
 $$

  $$
 \raise -10pt\hbox{ \begin{picture}(20,40)
       \put(0,6){\circle{3}} \put(0,16){\circle*{3}} \put(0,26){\circle*{3}} 
       \put(0,36){\circle*{3}}
       \put(0,6){\line(0,1){10}}
       \put(0,16){\line(0,1){10}}
       \put(0,26){\line(0,1){10}}
       \put(3,5){$\scriptscriptstyle 3$} \put(3,17){$\scriptscriptstyle 1$} 
       \put(3,27){$\scriptscriptstyle 2$}  \put(3,37){$\scriptscriptstyle 4$}
       \end{picture}}
 \raise -10pt\hbox{ \begin{picture}(20,40)
       \put(0,6){\circle{3}} \put(0,16){\circle*{3}} \put(0,26){\circle*{3}} 
       \put(0,36){\circle*{3}}
       \put(0,6){\line(0,1){10}}
       \put(0,16){\line(0,1){10}}
       \put(3,5){$\scriptscriptstyle 3$} \put(3,17){$\scriptscriptstyle 1$} 
       \put(3,27){$\scriptscriptstyle 4$}  \put(3,37){$\scriptscriptstyle 2$}
    \rouge{ \put(-1,26){\line(0,1){10}}  \put(1,26){\line(0,1){10}} }
       \end{picture}}
 \raise -10pt\hbox{ \begin{picture}(20,40)
       \put(0,6){\circle{3}} \put(0,16){\circle*{3}} \put(0,26){\circle*{3}} 
       \put(0,36){\circle*{3}}
       \put(0,6){\line(0,1){10}}
       \put(0,16){\line(0,1){10}}
       \put(0,26){\line(0,1){10}}
       \put(3,5){$\scriptscriptstyle 3$} \put(3,17){$\scriptscriptstyle 2$} 
       \put(3,27){$\scriptscriptstyle 1$}  \put(3,37){$\scriptscriptstyle 4$}
       \end{picture}}
 \raise -10pt\hbox{ \begin{picture}(20,40)
       \put(0,6){\circle{3}} \put(0,16){\circle*{3}} \put(0,26){\circle*{3}} 
       \put(0,36){\circle*{3}}
       \put(0,6){\line(0,1){10}}
       \put(0,16){\line(0,1){10}}
       \put(3,5){$\scriptscriptstyle 3$} \put(3,17){$\scriptscriptstyle 2$} 
       \put(3,27){$\scriptscriptstyle 4$}  \put(3,37){$\scriptscriptstyle 1$}
    \rouge{ \put(-1,26){\line(0,1){10}}  \put(1,26){\line(0,1){10}}}
       \end{picture}}
 \raise -10pt\hbox{ \begin{picture}(20,40)
       \put(0,6){\circle{3}} \put(0,16){\circle*{3}} \put(0,26){\circle*{3}} 
       \put(0,36){\circle*{3}}
       \put(0,6){\line(0,1){10}}
       \put(0,26){\line(0,1){10}}
       \put(3,5){$\scriptscriptstyle 3$} \put(3,17){$\scriptscriptstyle 4$} 
       \put(3,27){$\scriptscriptstyle 1$}  \put(3,37){$\scriptscriptstyle 2$}
     \rouge{   \put(-1,16){\line(0,1){10}}    \put(1,16){\line(0,1){10}}  }
      \end{picture}}
 \raise -10pt\hbox{ \begin{picture}(20,40)
       \put(0,6){\circle{3}} \put(0,16){\circle*{3}} \put(0,26){\circle*{3}} 
       \put(0,36){\circle*{3}}
       \put(0,6){\line(0,1){10}}
      \put(3,5){$\scriptscriptstyle 3$} \put(3,17){$\scriptscriptstyle 4$} 
      \put(3,27){$\scriptscriptstyle 2$}  \put(3,37){$\scriptscriptstyle 1$}
     \rouge{   \put(-1,16){\line(0,1){10}}  \put(1,16){\line(0,1){10}}   \put(-1,26){\line(0,1){10}}   \put(1,26){\line(0,1){10}}  }
       \end{picture}}
 \raise -10pt\hbox{ \begin{picture}(20,40)
       \put(0,6){\circle{3}} \put(0,16){\circle*{3}} \put(0,26){\circle*{3}} 
       \put(0,36){\circle*{3}}
       \put(0,16){\line(0,1){10}}
       \put(0,26){\line(0,1){10}}
       \put(3,5){$\scriptscriptstyle 4$} \put(3,17){$\scriptscriptstyle 1$} 
       \put(3,27){$\scriptscriptstyle 2$}  \put(3,37){$\scriptscriptstyle 3$}
     \rouge{  \put(-1,6){\line(0,1){10}}   \put(1,6){\line(0,1){10}} }
      \end{picture}}
 \raise -10pt\hbox{ \begin{picture}(20,40)
       \put(0,6){\circle{3}} \put(0,16){\circle*{3}} \put(0,26){\circle*{3}} 
       \put(0,36){\circle*{3}}
       \put(0,16){\line(0,1){10}}
       \put(3,5){$\scriptscriptstyle 4$} \put(3,17){$\scriptscriptstyle 1$} 
       \put(3,27){$\scriptscriptstyle 3$}  \put(3,37){$\scriptscriptstyle 2$}
     \rouge{  \put(-1,6){\line(0,1){10}}  \put(1,6){\line(0,1){10}}   \put(-1,26){\line(0,1){10}}   \put(1,26){\line(0,1){10}} }
      \end{picture}}
 \raise -10pt\hbox{ \begin{picture}(20,40)
       \put(0,6){\circle{3}} \put(0,16){\circle*{3}} \put(0,26){\circle*{3}} 
       \put(0,36){\circle*{3}}
       \put(0,16){\line(0,1){10}}
       \put(0,26){\line(0,1){10}}
       \put(3,5){$\scriptscriptstyle 4$} \put(3,17){$\scriptscriptstyle 2$} 
       \put(3,27){$\scriptscriptstyle 1$}  \put(3,37){$\scriptscriptstyle 3$}
     \rouge{  \put(-1,6){\line(0,1){10}}  \put(1,6){\line(0,1){10}} }
       \end{picture}}
 \raise -10pt\hbox{ \begin{picture}(20,40)
       \put(0,6){\circle{3}} \put(0,16){\circle*{3}} \put(0,26){\circle*{3}} 
       \put(0,36){\circle*{3}}
       \put(0,16){\line(0,1){10}}
       \put(3,5){$\scriptscriptstyle 4$} \put(3,17){$\scriptscriptstyle 2$} 
       \put(3,27){$\scriptscriptstyle 3$}  \put(3,37){$\scriptscriptstyle 1$}
     \rouge{  \put(-1,26){\line(0,1){10}}  \put(1,26){\line(0,1){10}}   \put(-1,6){\line(0,1){10}}  \put(1,6){\line(0,1){10}} }
       \end{picture}}
 \raise -10pt\hbox{ \begin{picture}(20,40)
       \put(0,6){\circle{3}} \put(0,16){\circle*{3}} \put(0,26){\circle*{3}} 
       \put(0,36){\circle*{3}}
       \put(0,26){\line(0,1){10}}
       \put(3,5){$\scriptscriptstyle 4$} \put(3,17){$\scriptscriptstyle 3$} 
       \put(3,27){$\scriptscriptstyle 1$}  \put(3,37){$\scriptscriptstyle 2$}
     \rouge{  \put(-1,6){\line(0,1){10}}   \put(1,6){\line(0,1){10}}    \put(-1,16){\line(0,1){10}}   \put(1,16){\line(0,1){10}}  }
       \end{picture}}
 \raise -10pt\hbox{ \begin{picture}(20,40)
       \put(0,6){\circle{3}} \put(0,16){\circle*{3}} \put(0,26){\circle*{3}} 
       \put(0,36){\circle*{3}}
     \rouge{  \put(-1,6){\line(0,1){10}}  \put(1,6){\line(0,1){10}} 
                   \put(-1,16){\line(0,1){10}}   \put(1,16){\line(0,1){10}}
                   \put(-1,26){\line(0,1){10}}  \put(1,26){\line(0,1){10}} }
       \put(3,5){$\scriptscriptstyle 4$} \put(3,17){$\scriptscriptstyle 3$} 
       \put(3,27){$\scriptscriptstyle 2$}  \put(3,37){$\scriptscriptstyle 1$}
       \end{picture}}
 $$

  $$
 \raise -10pt\hbox{ \begin{picture}(35,30)
       \put(0,6){\circle{3}} \put(0,16){\circle*{3}} \put(-10,26){\circle*{3}} 
       \put(10,26){\circle*{3}}
       \put(0,6){\line(0,1){10}}
       \put(0,16){\line(1,1){10}}
       \put(0,16){\line(-1,1){10}}
       \put(3,5){$\scriptscriptstyle 1$} \put(3,15){$\scriptscriptstyle 2$} 
       \put(-13,30){$\scriptscriptstyle 3$}  \put(10,30){$\scriptscriptstyle 4$}
      \end{picture}}
 \raise -10pt\hbox{ \begin{picture}(30,30)
       \put(0,6){\circle{3}} \put(0,16){\circle*{3}} \put(-10,26){\circle*{3}} 
       \put(10,26){\circle*{3}}
       \put(0,6){\line(0,1){10}}
     \rouge{    \put(-1,16){\line(1,1){10}}   \put(1,16){\line(1,1){10}}
                     \put(-1,16){\line(-1,1){10}}   \put(1,16){\line(-1,1){10}}  }
       \put(3,5){$\scriptscriptstyle 1$} \put(3,15){$\scriptscriptstyle 3$} 
       \put(-13,30){$\scriptscriptstyle 2$}  \put(10,30){$\scriptscriptstyle 4$}
       \end{picture}}
 \raise -10pt\hbox{ \begin{picture}(30,30)
       \put(0,6){\circle{3}} \put(0,16){\circle*{3}} \put(-10,26){\circle*{3}} 
       \put(10,26){\circle*{3}}
       \put(0,6){\line(0,1){10}}
     \rouge{   \put(-1,16){\line(1,1){10}}  \put(1,16){\line(1,1){10}}
                    \put(-1,16){\line(-1,1){10}}   \put(1,16){\line(-1,1){10}}  }
       \put(3,5){$\scriptscriptstyle 1$} \put(3,15){$\scriptscriptstyle 4$} 
       \put(-13,30){$\scriptscriptstyle 2$}  \put(10,30){$\scriptscriptstyle 3$}
      \end{picture}}
 \raise -10pt\hbox{ \begin{picture}(30,30)
       \put(0,6){\circle{3}} \put(0,16){\circle*{3}} \put(-10,26){\circle*{3}} 
       \put(10,26){\circle*{3}}
       \put(0,6){\line(0,1){10}}
       \put(0,16){\line(1,1){10}}
       \put(0,16){\line(-1,1){10}}
       \put(3,5){$\scriptscriptstyle 2$} \put(3,15){$\scriptscriptstyle 1$} 
       \put(-13,30){$\scriptscriptstyle 3$}  \put(10,30){$\scriptscriptstyle 4$}
      \end{picture}}
 \raise -10pt\hbox{ \begin{picture}(30,30)
       \put(0,6){\circle{3}} \put(0,16){\circle*{3}} \put(-10,26){\circle*{3}} 
       \put(10,26){\circle*{3}}
       \put(0,6){\line(0,1){10}}
     \rouge{   \put(-1,16){\line(1,1){10}}  \put(1,16){\line(1,1){10}}
                    \put(-1,16){\line(-1,1){10}}   \put(1,16){\line(-1,1){10}}  }
       \put(3,5){$\scriptscriptstyle 2$} \put(3,15){$\scriptscriptstyle 3$} 
       \put(-13,30){$\scriptscriptstyle 1$}  \put(10,30){$\scriptscriptstyle 4$}
      \end{picture}}
 \raise -10pt\hbox{ \begin{picture}(30,30)
       \put(0,6){\circle{3}} \put(0,16){\circle*{3}} \put(-10,26){\circle*{3}} 
       \put(10,26){\circle*{3}}
       \put(0,6){\line(0,1){10}}
     \rouge{   \put(-1,16){\line(1,1){10}}  \put(1,16){\line(1,1){10}}
                    \put(-1,16){\line(-1,1){10}}   \put(1,16){\line(-1,1){10}}  }
      \put(3,5){$\scriptscriptstyle 2$} \put(3,15){$\scriptscriptstyle 4$} 
      \put(-13,30){$\scriptscriptstyle 1$}  \put(10,30){$\scriptscriptstyle 3$}
      \end{picture}}
$$

  $$
 \raise -10pt\hbox{ \begin{picture}(30,30)
       \put(0,6){\circle{3}} \put(0,16){\circle*{3}} \put(-10,26){\circle*{3}} 
       \put(10,26){\circle*{3}}
       \put(0,6){\line(0,1){10}}
       \put(0,16){\line(1,1){10}}
       \put(0,16){\line(-1,1){10}}
       \put(3,5){$\scriptscriptstyle 3$} \put(3,15){$\scriptscriptstyle 1$} 
       \put(-13,30){$\scriptscriptstyle 2$}  \put(10,30){$\scriptscriptstyle 4$}
      \end{picture}}
 \raise -10pt\hbox{ \begin{picture}(30,30)
       \put(0,6){\circle{3}} \put(0,16){\circle*{3}} \put(-10,26){\circle*{3}} 
       \put(10,26){\circle*{3}}
     \rouge{   \put(-1,6){\line(0,1){10}}   \put(1,6){\line(0,1){10}}
                   \put(-1,16){\line(1,1){10}}   \put(1,16){\line(1,1){10}}
                   \put(-1,16){\line(-1,1){10}}    \put(1,16){\line(-1,1){10}}}
       \put(3,5){$\scriptscriptstyle 3$} \put(3,15){$\scriptscriptstyle 2$} 
       \put(-13,30){$\scriptscriptstyle 1$}  \put(10,30){$\scriptscriptstyle 4$}
      \end{picture}}
 \raise -10pt\hbox{ \begin{picture}(30,30)
       \put(0,6){\circle{3}} \put(0,16){\circle*{3}} \put(-10,26){\circle*{3}} 
       \put(10,26){\circle*{3}}
       \put(0,6){\line(0,1){10}}
     \rouge{   \put(-1,16){\line(1,1){10}}  \put(1,16){\line(1,1){10}}
                    \put(-1,16){\line(-1,1){10}}   \put(1,16){\line(-1,1){10}}  }
       \put(3,5){$\scriptscriptstyle 3$} \put(3,15){$\scriptscriptstyle 4$} 
       \put(-13,30){$\scriptscriptstyle 1$}  \put(10,30){$\scriptscriptstyle 2$}
      \end{picture}}
 \raise -10pt\hbox{ \begin{picture}(30,30)
       \put(0,6){\circle{3}} \put(0,16){\circle*{3}} \put(-10,26){\circle*{3}} 
       \put(10,26){\circle*{3}}
     \rouge{  \put(-1,6){\line(0,1){10}}  \put(1,6){\line(0,1){10}} }
       \put(0,16){\line(1,1){10}}
       \put(0,16){\line(-1,1){10}}
       \put(3,5){$\scriptscriptstyle 4$} \put(3,15){$\scriptscriptstyle 1$} 
       \put(-13,30){$\scriptscriptstyle 2$}  \put(10,30){$\scriptscriptstyle 3$}
      \end{picture}}
 \raise -10pt\hbox{ \begin{picture}(30,30)
       \put(0,6){\circle{3}} \put(0,16){\circle*{3}} \put(-10,26){\circle*{3}} 
       \put(10,26){\circle*{3}}
     \rouge{   \put(-1,6){\line(0,1){10}}   \put(1,6){\line(0,1){10}}
                   \put(-1,16){\line(1,1){10}}   \put(1,16){\line(1,1){10}}
                   \put(-1,16){\line(-1,1){10}}    \put(1,16){\line(-1,1){10}}}
       \put(3,5){$\scriptscriptstyle 4$} \put(3,15){$\scriptscriptstyle 2$} 
       \put(-13,30){$\scriptscriptstyle 1$}  \put(10,30){$\scriptscriptstyle 3$}
      \end{picture}}
 \raise -10pt\hbox{ \begin{picture}(30,30)
       \put(0,6){\circle{3}} \put(0,16){\circle*{3}} \put(-10,26){\circle*{3}} 
       \put(10,26){\circle*{3}}
     \rouge{   \put(-1,6){\line(0,1){10}}   \put(1,6){\line(0,1){10}}
                   \put(-1,16){\line(1,1){10}}   \put(1,16){\line(1,1){10}}
                   \put(-1,16){\line(-1,1){10}}    \put(1,16){\line(-1,1){10}}}
       \put(3,5){$\scriptscriptstyle 4$} \put(3,15){$\scriptscriptstyle 3$} 
       \put(-13,30){$\scriptscriptstyle 1$}  \put(10,30){$\scriptscriptstyle 2$}
      \end{picture}}
$$

  $$
 \raise -10pt\hbox{ \begin{picture}(30,30)
       \put(0,6){\circle{3}} \put(-10,16){\circle*{3}} \put(10,16){\circle*{3}} 
       \put(10,26){\circle*{3}}
       \put(0,6){\line(-1,1){10}}
       \put(0,6){\line(1,1){10}}
       \put(10,16){\line(0,1){10}}
       \put(-8,5){$\scriptscriptstyle 1$} \put(-13,20){$\scriptscriptstyle 2$}  
       \put(3,15){$\scriptscriptstyle 3$}  \put(10,30){$\scriptscriptstyle 4$}
      \end{picture}}
 \raise -10pt\hbox{ \begin{picture}(30,30)
       \put(0,6){\circle{3}} \put(-10,16){\circle*{3}} \put(10,16){\circle*{3}} 
       \put(10,26){\circle*{3}}
       \put(0,6){\line(-1,1){10}}
       \put(0,6){\line(1,1){10}}
     \rouge{   \put(9,16){\line(0,1){10}}    \put(11,16){\line(0,1){10}} }
       \put(-8,5){$\scriptscriptstyle 1$} \put(-13,20){$\scriptscriptstyle 2$}  
       \put(3,15){$\scriptscriptstyle 4$}  \put(10,30){$\scriptscriptstyle 3$}
      \end{picture}}
 \raise -10pt\hbox{ \begin{picture}(30,30)
       \put(0,6){\circle{3}} \put(-10,16){\circle*{3}} \put(10,16){\circle*{3}} 
       \put(10,26){\circle*{3}}
       \put(0,6){\line(-1,1){10}}
       \put(0,6){\line(1,1){10}}
       \put(10,16){\line(0,1){10}}
       \put(-8,5){$\scriptscriptstyle 1$} \put(-13,20){$\scriptscriptstyle 3$}  
       \put(3,15){$\scriptscriptstyle 2$}  \put(10,30){$\scriptscriptstyle 4$}
      \end{picture}}
 \raise -10pt\hbox{ \begin{picture}(30,30)
       \put(0,6){\circle{3}} \put(-10,16){\circle*{3}} \put(10,16){\circle*{3}} 
       \put(10,26){\circle*{3}}
       \put(0,6){\line(-1,1){10}}
       \put(0,6){\line(1,1){10}}
     \rouge{   \put(9,16){\line(0,1){10}}    \put(11,16){\line(0,1){10}} }
       \put(-8,5){$\scriptscriptstyle 1$} \put(-13,20){$\scriptscriptstyle 3$}  
       \put(3,15){$\scriptscriptstyle 4$}  \put(10,30){$\scriptscriptstyle 2$}
      \end{picture}}
 \raise -10pt\hbox{ \begin{picture}(30,30)
       \put(0,6){\circle{3}} \put(-10,16){\circle*{3}} \put(10,16){\circle*{3}} 
       \put(-10,26){\circle*{3}}
       \put(0,6){\line(-1,1){10}}
       \put(0,6){\line(1,1){10}}
       \put(-10,16){\line(0,1){10}}
       \put(-8,5){$\scriptscriptstyle 1$} \put(-7,15){$\scriptscriptstyle 2$}  
       \put(10,18){$\scriptscriptstyle 4$}  \put(-13,30){$\scriptscriptstyle 3$}
      \end{picture}}
 \raise -10pt\hbox{ \begin{picture}(30,30)
       \put(0,6){\circle{3}} \put(-10,16){\circle*{3}} \put(10,16){\circle*{3}} 
       \put(-10,26){\circle*{3}}
    \rouge{  \put(-1,6){\line(-1,1){10}} \put(1,6){\line(-1,1){10}}
                 \put(-1,6){\line(1,1){10}}  \put(1,6){\line(1,1){10}}
                 \put(-11,16){\line(0,1){10}}   \put(-9,16){\line(0,1){10}} }
       \put(-8,5){$\scriptscriptstyle 1$} \put(-7,15){$\scriptscriptstyle 3$}  
       \put(10,18){$\scriptscriptstyle 4$}  \put(-13,30){$\scriptscriptstyle 2$}
      \end{picture}}
$$

  $$
 \raise -10pt\hbox{ \begin{picture}(30,30)
       \put(0,6){\circle{3}} \put(-10,16){\circle*{3}} \put(10,16){\circle*{3}} 
       \put(10,26){\circle*{3}}
     \rouge{   \put(-1,6){\line(-1,1){10}}   \put(1,6){\line(-1,1){10}}
                    \put(-1,6){\line(1,1){10}}   \put(1,6){\line(1,1){10}} }
       \put(10,16){\line(0,1){10}}
       \put(-8,5){$\scriptscriptstyle 2$} \put(-13,20){$\scriptscriptstyle 1$}  
       \put(3,15){$\scriptscriptstyle 3$}  \put(10,30){$\scriptscriptstyle 4$}
      \end{picture}}
 \raise -10pt\hbox{ \begin{picture}(30,30)
       \put(0,6){\circle{3}} \put(-10,16){\circle*{3}} \put(10,16){\circle*{3}} 
       \put(10,26){\circle*{3}}
      \rouge{   \put(-1,6){\line(-1,1){10}}   \put(1,6){\line(-1,1){10}}
                    \put(-1,6){\line(1,1){10}}   \put(1,6){\line(1,1){10}}
                    \put(9,16){\line(0,1){10}}   \put(11,16){\line(0,1){10}} }
       \put(-8,5){$\scriptscriptstyle 2$} \put(-13,20){$\scriptscriptstyle 1$}  
       \put(3,15){$\scriptscriptstyle 4$}  \put(10,30){$\scriptscriptstyle 3$}
      \end{picture}}
 \raise -10pt\hbox{ \begin{picture}(30,30)
       \put(0,6){\circle{3}} \put(-10,16){\circle*{3}} \put(10,16){\circle*{3}} 
       \put(10,26){\circle*{3}}
       \put(0,6){\line(-1,1){10}}
       \put(0,6){\line(1,1){10}}
       \put(10,16){\line(0,1){10}}
       \put(-8,5){$\scriptscriptstyle 2$} \put(-13,20){$\scriptscriptstyle 3$}  
       \put(3,15){$\scriptscriptstyle 1$}  \put(10,30){$\scriptscriptstyle 4$}
      \end{picture}}
 \raise -10pt\hbox{ \begin{picture}(30,30)
       \put(0,6){\circle{3}} \put(-10,16){\circle*{3}} \put(10,16){\circle*{3}} 
       \put(10,26){\circle*{3}}
       \put(0,6){\line(-1,1){10}}
       \put(0,6){\line(1,1){10}}
      \rouge{     \put(9,16){\line(0,1){10}}     \put(11,16){\line(0,1){10}} }
       \put(-8,5){$\scriptscriptstyle 2$} \put(-13,20){$\scriptscriptstyle 3$}  
       \put(3,15){$\scriptscriptstyle 4$}  \put(10,30){$\scriptscriptstyle 1$}
      \end{picture}}
 \raise -10pt\hbox{ \begin{picture}(30,30)
       \put(0,6){\circle{3}} \put(-10,16){\circle*{3}} \put(10,16){\circle*{3}} 
       \put(-10,26){\circle*{3}}
       \put(0,6){\line(-1,1){10}}
       \put(0,6){\line(1,1){10}}
       \put(-10,16){\line(0,1){10}}
       \put(-8,5){$\scriptscriptstyle 2$} \put(-7,15){$\scriptscriptstyle 1$}  
       \put(10,18){$\scriptscriptstyle 4$}  \put(-13,30){$\scriptscriptstyle 3$}
      \end{picture}}
 \raise -10pt\hbox{ \begin{picture}(30,30)
       \put(0,6){\circle{3}} \put(-10,16){\circle*{3}} \put(10,16){\circle*{3}} 
       \put(-10,26){\circle*{3}}
    \rouge{  \put(-1,6){\line(-1,1){10}} \put(1,6){\line(-1,1){10}}
                 \put(-1,6){\line(1,1){10}}  \put(1,6){\line(1,1){10}}
                 \put(-11,16){\line(0,1){10}}   \put(-9,16){\line(0,1){10}} }
       \put(-8,5){$\scriptscriptstyle 2$} \put(-7,15){$\scriptscriptstyle 3$}  
       \put(10,18){$\scriptscriptstyle 4$}  \put(-13,30){$\scriptscriptstyle 1$}
      \end{picture}}
$$

  $$
 \raise -10pt\hbox{ \begin{picture}(30,30)
       \put(0,6){\circle{3}} \put(-10,16){\circle*{3}} \put(10,16){\circle*{3}} 
       \put(10,26){\circle*{3}}
     \rouge{   \put(-1,6){\line(-1,1){10}}   \put(1,6){\line(-1,1){10}}
                    \put(-1,6){\line(1,1){10}}   \put(1,6){\line(1,1){10}} }
       \put(10,16){\line(0,1){10}}
       \put(-8,5){$\scriptscriptstyle 3$} \put(-13,20){$\scriptscriptstyle 1$}  
       \put(3,15){$\scriptscriptstyle 2$}  \put(10,30){$\scriptscriptstyle 4$}
      \end{picture}}
 \raise -10pt\hbox{ \begin{picture}(30,30)
       \put(0,6){\circle{3}} \put(-10,16){\circle*{3}} \put(10,16){\circle*{3}} 
       \put(10,26){\circle*{3}}
      \rouge{   \put(-1,6){\line(-1,1){10}}   \put(1,6){\line(-1,1){10}}
                    \put(-1,6){\line(1,1){10}}   \put(1,6){\line(1,1){10}}
                    \put(9,16){\line(0,1){10}}   \put(11,16){\line(0,1){10}} }
       \put(-8,5){$\scriptscriptstyle 3$} \put(-13,20){$\scriptscriptstyle 1$}  
       \put(3,15){$\scriptscriptstyle 4$}  \put(10,30){$\scriptscriptstyle 2$}
      \end{picture}}
 \raise -10pt\hbox{ \begin{picture}(30,30)
       \put(0,6){\circle{3}} \put(-10,16){\circle*{3}} \put(10,16){\circle*{3}} 
       \put(10,26){\circle*{3}}
     \rouge{   \put(-1,6){\line(-1,1){10}}   \put(1,6){\line(-1,1){10}}
                    \put(-1,6){\line(1,1){10}}   \put(1,6){\line(1,1){10}} }
       \put(10,16){\line(0,1){10}}
       \put(-8,5){$\scriptscriptstyle 3$} \put(-13,20){$\scriptscriptstyle 2$}  
       \put(3,15){$\scriptscriptstyle 1$}  \put(10,30){$\scriptscriptstyle 4$}
      \end{picture}}
 \raise -10pt\hbox{ \begin{picture}(30,30)
       \put(0,6){\circle{3}} \put(-10,16){\circle*{3}} \put(10,16){\circle*{3}} 
       \put(10,26){\circle*{3}}
      \rouge{   \put(-1,6){\line(-1,1){10}}   \put(1,6){\line(-1,1){10}}
                    \put(-1,6){\line(1,1){10}}   \put(1,6){\line(1,1){10}}
                    \put(9,16){\line(0,1){10}}   \put(11,16){\line(0,1){10}} }
       \put(-8,5){$\scriptscriptstyle 3$} \put(-13,20){$\scriptscriptstyle 2$}  
       \put(3,15){$\scriptscriptstyle 4$}  \put(10,30){$\scriptscriptstyle 1$}
      \end{picture}}
 \raise -10pt\hbox{ \begin{picture}(30,30)
       \put(0,6){\circle{3}} \put(-10,16){\circle*{3}} \put(10,16){\circle*{3}} 
       \put(-10,26){\circle*{3}}
     \rouge{   \put(-1,6){\line(-1,1){10}}   \put(1,6){\line(-1,1){10}}
                    \put(-1,6){\line(1,1){10}}   \put(1,6){\line(1,1){10}} }
       \put(-10,16){\line(0,1){10}}
       \put(-8,5){$\scriptscriptstyle 3$} \put(-7,15){$\scriptscriptstyle 1$}  
       \put(10,18){$\scriptscriptstyle 4$}  \put(-13,30){$\scriptscriptstyle 2$}
      \end{picture}}
 \raise -10pt\hbox{ \begin{picture}(30,30)
       \put(0,6){\circle{3}} \put(-10,16){\circle*{3}} \put(10,16){\circle*{3}} 
       \put(-10,26){\circle*{3}}
    \rouge{  \put(-1,6){\line(-1,1){10}} \put(1,6){\line(-1,1){10}}
                 \put(-1,6){\line(1,1){10}}  \put(1,6){\line(1,1){10}}
                 \put(-11,16){\line(0,1){10}}   \put(-9,16){\line(0,1){10}} }
       \put(-8,5){$\scriptscriptstyle 3$} \put(-7,15){$\scriptscriptstyle 2$}  
       \put(10,18){$\scriptscriptstyle 4$}  \put(-13,30){$\scriptscriptstyle 1$}
      \end{picture}}
$$

  $$
 \raise -10pt\hbox{ \begin{picture}(30,30)
       \put(0,6){\circle{3}} \put(-10,16){\circle*{3}} \put(10,16){\circle*{3}} 
       \put(10,26){\circle*{3}}
     \rouge{   \put(-1,6){\line(-1,1){10}}   \put(1,6){\line(-1,1){10}}
                    \put(-1,6){\line(1,1){10}}   \put(1,6){\line(1,1){10}} }
       \put(10,16){\line(0,1){10}}
       \put(-8,5){$\scriptscriptstyle 4$} \put(-13,20){$\scriptscriptstyle 1$}  
       \put(3,15){$\scriptscriptstyle 2$}  \put(10,30){$\scriptscriptstyle 3$}
      \end{picture}}
 \raise -10pt\hbox{ \begin{picture}(30,30)
       \put(0,6){\circle{3}} \put(-10,16){\circle*{3}} \put(10,16){\circle*{3}} 
       \put(10,26){\circle*{3}}
      \rouge{   \put(-1,6){\line(-1,1){10}}   \put(1,6){\line(-1,1){10}}
                    \put(-1,6){\line(1,1){10}}   \put(1,6){\line(1,1){10}}
                    \put(9,16){\line(0,1){10}}   \put(11,16){\line(0,1){10}} }
       \put(-8,5){$\scriptscriptstyle 4$} \put(-13,20){$\scriptscriptstyle 1$}  
       \put(3,15){$\scriptscriptstyle 3$}  \put(10,30){$\scriptscriptstyle 2$}
      \end{picture}}
 \raise -10pt\hbox{ \begin{picture}(30,30)
       \put(0,6){\circle{3}} \put(-10,16){\circle*{3}} \put(10,16){\circle*{3}} 
       \put(10,26){\circle*{3}}
     \rouge{   \put(-1,6){\line(-1,1){10}}   \put(1,6){\line(-1,1){10}}
                    \put(-1,6){\line(1,1){10}}   \put(1,6){\line(1,1){10}} }
       \put(10,16){\line(0,1){10}}
       \put(-8,5){$\scriptscriptstyle 4$} \put(-13,20){$\scriptscriptstyle 2$}  
       \put(3,15){$\scriptscriptstyle 1$}  \put(10,30){$\scriptscriptstyle 3$}
      \end{picture}}
 \raise -10pt\hbox{ \begin{picture}(30,30)
       \put(0,6){\circle{3}} \put(-10,16){\circle*{3}} \put(10,16){\circle*{3}} 
       \put(10,26){\circle*{3}}
     \rouge{   \put(-1,6){\line(-1,1){10}}   \put(1,6){\line(-1,1){10}}
                    \put(-1,6){\line(1,1){10}}   \put(1,6){\line(1,1){10}}
                    \put(9,16){\line(0,1){10}}   \put(11,16){\line(0,1){10}} }
       \put(-8,5){$\scriptscriptstyle 4$} \put(-13,20){$\scriptscriptstyle 2$}  
       \put(3,15){$\scriptscriptstyle 3$}  \put(10,30){$\scriptscriptstyle 1$}
      \end{picture}}
 \raise -10pt\hbox{ \begin{picture}(30,30)
       \put(0,6){\circle{3}} \put(-10,16){\circle*{3}} \put(10,16){\circle*{2}} 
       \put(-10,26){\circle*{3}}
     \rouge{   \put(-1,6){\line(-1,1){10}}   \put(1,6){\line(-1,1){10}}
                    \put(-1,6){\line(1,1){10}}   \put(1,6){\line(1,1){10}} }
      \put(-10,16){\line(0,1){10}}
       \put(-8,5){$\scriptscriptstyle 4$} \put(-7,15){$\scriptscriptstyle 1$}  
       \put(10,18){$\scriptscriptstyle 3$}  \put(-13,30){$\scriptscriptstyle 2$}
      \end{picture}}
 \raise -10pt\hbox{ \begin{picture}(30,30)
       \put(0,6){\circle{3}} \put(-10,16){\circle*{3}} \put(10,16){\circle*{3}} 
       \put(-10,26){\circle*{3}}
    \rouge{  \put(-1,6){\line(-1,1){10}} \put(1,6){\line(-1,1){10}}
                 \put(-1,6){\line(1,1){10}}  \put(1,6){\line(1,1){10}}
                 \put(-11,16){\line(0,1){10}}   \put(-9,16){\line(0,1){10}} }
       \put(-8,5){$\scriptscriptstyle 4$} \put(-7,15){$\scriptscriptstyle 2$}  
       \put(10,18){$\scriptscriptstyle 3$}  \put(-13,30){$\scriptscriptstyle 1$}
      \end{picture}}
$$

  $$
 \raise -10pt\hbox{ \begin{picture}(35,30)
       \put(0,6){\circle{3}} \put(-10,16){\circle*{3}} \put(10,16){\circle*{3}} 
       \put(0,16){\circle*{3}}
       \put(0,6){\line(-1,1){10}}
       \put(0,6){\line(1,1){10}}
       \put(0,6){\line(0,1){10}}
       \put(-8,5){$\scriptscriptstyle 1$} \put(-13,20){$\scriptscriptstyle 2$}  
       \put(-2,20){$\scriptscriptstyle 3$}  \put(10,20){$\scriptscriptstyle 4$}
      \end{picture}}
 \raise -10pt\hbox{ \begin{picture}(35,30)
       \put(0,6){\circle{3}} \put(-10,16){\circle*{3}} \put(10,16){\circle*{3}} 
       \put(0,16){\circle*{3}}
     \rouge{   \put(-1,6){\line(-1,1){10}}  \put(1,6){\line(-1,1){10}}
                   \put(-1,6){\line(1,1){10}}    \put(1,6){\line(1,1){10}}
                   \put(-1,6){\line(0,1){10}}    \put(1,6){\line(0,1){10}}  }
       \put(-8,5){$\scriptscriptstyle 2$} \put(-13,20){$\scriptscriptstyle 1$}  
       \put(-2,20){$\scriptscriptstyle 3$}  \put(10,20){$\scriptscriptstyle 4$}
      \end{picture}}
 \raise -10pt\hbox{ \begin{picture}(35,30)
       \put(0,6){\circle{3}} \put(-10,16){\circle*{3}} \put(10,16){\circle*{3}} 
       \put(0,16){\circle*{3}}
     \rouge{   \put(-1,6){\line(-1,1){10}}  \put(1,6){\line(-1,1){10}}
                   \put(-1,6){\line(1,1){10}}    \put(1,6){\line(1,1){10}}
                   \put(-1,6){\line(0,1){10}}    \put(1,6){\line(0,1){10}}  }
       \put(-8,5){$\scriptscriptstyle 3$} \put(-13,20){$\scriptscriptstyle 1$}  
       \put(-2,20){$\scriptscriptstyle 2$}  \put(10,20){$\scriptscriptstyle 4$}
      \end{picture}}
 \raise -10pt\hbox{ \begin{picture}(35,30)
       \put(0,6){\circle{3}} \put(-10,16){\circle*{3}} \put(10,16){\circle*{3}} 
       \put(0,16){\circle*{3}}
     \rouge{   \put(-1,6){\line(-1,1){10}}  \put(1,6){\line(-1,1){10}}
                   \put(-1,6){\line(1,1){10}}    \put(1,6){\line(1,1){10}}
                   \put(-1,6){\line(0,1){10}}    \put(1,6){\line(0,1){10}}  }
       \put(-8,5){$\scriptscriptstyle 4$} \put(-13,20){$\scriptscriptstyle 1$}  
       \put(-2,20){$\scriptscriptstyle 2$}  \put(10,20){$\scriptscriptstyle 3$}
      \end{picture}}
$$

This decomposes each tree of $\PP [S]$ in a unique way. 
A more careful look at this decomposition gives us the following claim

\medskip
{\bf Claim 3} The subset of trees $\XXX[S]\subseteq \PP[S]$ that are all double red are characterized as follows. If $|S|=1$, then $\XXX[S] = \PP[S]$.
If $|S|>1$, then let $T=(r,\{T_1,T_2,\ldots,T_\ell\})$, let $r_i$ be the root of $T_i$ and assume $r_1<r_2<\cdots<r_\ell$.
We have $T\in\XXX[S]$ if and only if

(a) $T_i\in \XXX[S_i]$ for all $i\in\{1,2,\ldots,\ell\}$ and

(b) $r>r_1$ or ($r<r_1$ and $\ell>1$ and at least one $T_i$ for $i<\ell$ has size strictly bigger than 1). 

\medskip
\noindent
This gives us all the needed ingredients for the isomorphism in~(\ref{redblack}). For any $T\in \PP[S]$ the unique black and double red tree we obtain gives us a set partition $\phi\vdash S$ induced by the connected black components. On each black component we have a tree of type $A$. 
This decomposition also induces a double red tree on $\phi$ where each part $R\in\phi$ is identified with its maximal element. This gives us a bijection (of basis) between the trees in $\PP[S]$ and $\bigoplus_{\phi\vdash S} \XXX[\phi] \times \prod_{R\in\phi} ComMag[R]$. The inverse map is simply: connect the trees of type $A$ (identified to trees in $ComMag[R]$) according to $\XXX[\phi]$ from the root to the maximum of the tree it attaches to.

For example for
$$
T=\quad \raise -10pt
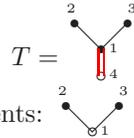
\hbox{ \begin{picture}(30,30)
       \put(0,6){\circle{3}} \put(0,16){\circle*{3}} \put(-10,26){\circle*{3}}
       \put(10,26){\circle*{3}}
       \rouge{       \put(-1,6){\line(0,1){10}}     \put(1,6){\line(0,1){10}}}
       \put(0,16){\line(1,1){10}}
       \put(0,16){\line(-1,1){10}}
       \put(3,5){$\scriptscriptstyle 4$} \put(3,15){$\scriptscriptstyle 1$}
       \put(-13,30){$\scriptscriptstyle 2$}  \put(10,30){$\scriptscriptstyle 3$}
      \end{picture}}
$$
we have two connected components:
$\quad  \raise -20pt\hbox{ \begin{picture}(15,10)
      \put(0,16){\circle{3}} \put(-10,26){\circle*{3}} \put(10,26){\circle*{3}}
       \put(0,16){\line(1,1){10}}
       \put(0,16){\line(-1,1){10}}
       \put(3,15){$\scriptscriptstyle 1$} \put(-13,30){$\scriptscriptstyle 2$}
       \put(10,30){$\scriptscriptstyle 3$}
      \end{picture}}
\in A[\{1,2,3\}]$ and
$\raise -2pt\hbox{ \begin{picture}(8,10)
       \put(0,6){\circle{3}}
       \put(3,5){$\scriptscriptstyle 4$}
      \end{picture}} \in A[\{4\}]$. This gives us the following set partition $\{\{1,2,3\},\{4\}\}$.
We labelled them $3$ and $4$ respectively. The double red edges  induce the following tree
$\raise -10pt\hbox{ \begin{picture}(8,20)
       \put(0,6){\circle{3}} \put(0,16){\circle*{3}}
     \rouge{  \put(-1,6){\line(0,1){10}}  \put(1,6){\line(0,1){10}}  }
       \put(3,5){$\scriptscriptstyle 4$} \put(3,17){$\scriptscriptstyle 3$}
      \end{picture}}
\in \XXX[\{3,4\}]$. It is clear how to reconstruct $T$ from this data knowing that we have to connect the root of 
$\quad  \raise -20pt\hbox{ \begin{picture}(15,10)
      \put(0,16){\circle{3}} \put(-10,26){\circle*{3}} \put(10,26){\circle*{3}}
       \put(0,16){\line(1,1){10}}
       \put(0,16){\line(-1,1){10}}
       \put(3,15){$\scriptscriptstyle 1$} \put(-13,30){$\scriptscriptstyle 2$}
       \put(10,30){$\scriptscriptstyle 3$}
      \end{picture}}$
 to the maximal element in the tree 
$\raise -10pt\hbox{ \begin{picture}(8,20)
       \put(0,6){\circle{3}} \put(0,16){\circle*{3}}
     \rouge{  \put(-1,6){\line(0,1){10}}  \put(1,6){\line(0,1){10}} }
       \put(3,5){$\scriptscriptstyle 4$} \put(3,17){$\scriptscriptstyle 3$}
      \end{picture}}$.

\medskip
\noindent

\subsection{Bimodule structure} Comment from Vladimir Dotsenko: the behavior of the double red trees suggests that the $\Sy$-module $\XXX$ itself can be decomposed as 
$\XXX=ComMag\circ \ZZZ$ for some $\Sy$-module  $\ZZZ$. Computation of the power series agrees with this statement. That would make $preLie$ into a free bimodule over $ComMag$.

\subsection{Comment on the splitting of $preLie$ and $NAP$} 
Since $preLie$ (resp. $NAP$) and $ComMag$ are operads it is natural to ask oneself if there could
 be an operad structure on $\XXX$ which would be compatible with the other operad structures under
  the isomorphism of $\Sy$-modules. Analogously, it is natural to ask for ``good triples of operads'' 
  $(\XXX, preLie, ComMag)$ (resp.\ $(\XXX,NAP, ComMag)$) analogous to the good triple $(Com, As, Lie)$ , cf. \ \cite{JLLgbto}.


\section{Comparison with the operad of dendriform algebras}

 We give another proof of Theorem \ref{mainthm} based on the comparison between the operad $preLie$ and the operads $Dend$ and $Dup$. First, we recall results about dendriform algebras and their parametrized version. Second, we recall results of \cite{JLLgbto} Chapter 5 about the operad $Dup$ encoding duplicial algebras. Finally we give an alternative proof to the injectivity of the map $ComMag \to preLie$ .
 
\subsection{Parametrized dendriform algebras} Let $\lambda\in \KK$ be a parameter. We define a 
\emph{parametrized dendriform algebra} as a vector space $A$ equipped with two binary operations 
$$\g : A\otimes A \to A \quad \mathrm{and } \quad \d : A\otimes A \to A $$
called the \emph{left operation} and the \emph{right operation} respectively, satisfying the following three relations
\begin{displaymath}
\left\{\begin{array}{rcl}
(x\g y)\g z &=& x\g (y\g z)+ \lambda x\g (y\d z)   , \\
(x\d y)\g z &=& x\d (y\g z) , \\
\lambda(x\g y)\d z + (x\d y)\d z&=& x\d (y\d z).
\end{array}
\right .
\end{displaymath}

For $\lambda=1$ it is the notion of \emph{dendriform algebra} introduced in \cite{JLLnote, JLLdig} and for $\lambda=0$ it is the notion of \emph{duplicial algebra} introduced in \cite{JLLgbto, Ronco10}. The relevant operads are denoted by $Dend$ and $Dup$ respectively.

We introduce  the operation $\{x,y\}$ given by
$$\{x,y\} := x\g y - y \d x\ ,$$
and the operation $x\,\square\, y$ given by
$$x\,\square\, y:= x\g y - x\d y.$$
In the dendriform case $\{x,y\}$ is known to be a pre-Lie product.

\begin{lemma}\label{commdiag} The following square of categories of algebras, resp.\ of operads (cf.\ \cite{LV}), where the upper horizontal arrow is induced by $\{x,y\}$, the left vertical arrow is induced by $x\,\square\, y$ and the two other arrows are symmetrization,
$$\xymatrix{
Dend\alg \ar[r]\ar[d] & preLie\alg \ar[d] && Dend & preLie\ar[l]\\
Mag\alg \ar[r] & ComMag\alg  && Mag\ar[u] & ComMag\ar[l]\ar[u]\\
}$$
is commutative.
\end{lemma} 

\begin{proo} The commutativity of the diagram follows from the following equalities:
\begin{multline*}
x\,\square\, y + y\,\square\, x =(x\g y - x\d y) + (y\g x - y\d x) = \\
(x\g y - y \d x) + (y \g x - x\d y)= \{x, y\} + \{y,x\}.
\end{multline*} 
\end{proo}

\subsection{On duplicial algebras} 
As mentioned above a duplicial algebra is defined by two binary operations $x\g y$ and $x\d y$ satisfying three relations:
\begin{displaymath}
\left\{\begin{array}{rcl}
(x\g y)\g z &=& x\g (y\g z)   , \\
(x\d y)\g z &=& x\d (y\g z) , \\
 (x\d y)\d z&=& x\d (y\d z).
\end{array}
\right .
\end{displaymath}

There is a notion of generalized bialgebra which is duplicial as an algebra and coassociative as a coalgebra. The primitive part is an algebra encoded by the operad $Mag$. The generating operation of the magmatic operad is given by $x\square y = x\g y - x \d y$ in terms of the duplicial operations. Since the ``triple of operads'' $(Com, Dup, Mag)$ is good (Corollary 5.2.6 of \cite{JLLgbto}), there is an isomorphism $Dup \cong As\circ Mag$, and therefore the map $Mag\to Dup$ is injective.

\begin{lemma}\label{MagtoDend} The map of ns operads $Mag\to Dend$ induced by the operation $x\,\square\, y:= x\g y - x\d y$ is injective.
\end{lemma}
\begin{proo} The free $\lambda\textrm{-}Dend$-algebra on a generator $x$, that is  
$\lambda\textrm{-}Dend(x)\cong \bigoplus_{n\geq 1}\lambda\textrm{-}Dend_{n}$, is isomorphic to 
$\bigoplus_{n\geq 1}\KK[PBT_{n+1}]$. The proof is the same as the proof for $\lambda =1$ given in \cite{JLLdig}. The inverse of this isomorphism, which we denote by $\varphi$, is obtained as follows:
$$\varphi(\arbreA) = x,\quad  \varphi(r\vee s) = \varphi(r)\d x \g \varphi(s)$$
where $r$ and $s$ are planar binary trees.

Letting $\lambda =1$, resp.\ $\lambda =0$, we get, for any $n$,  isomorphisms of vector spaces
$$Dend_{n} \cong \KK[PBT_{n+1}]\cong Dup_{n}.$$
We consider, in  $Dend(x)$ and $Dup(x)$ respectively, the subspace generated by $x$ under the operation $\square$. It give rise to the following commutative diagram (of graded vector spaces)
$$\xymatrix{
 & Mag \ar[dl]_{\square} \ar[dr]^{\square}& \\
 Dend &\cong & Dup
 }$$
Since the map  $Mag\to Dup$  is injective so is the map $Mag\to Dend$.
\end{proo}

\subsection{Alternative proof of Theorem \ref{mainthm}}  
Since the map of operads $ComMag \to Mag$ is obviously injective, Lemma \ref{MagtoDend} implies that the composite 
$ComMag \to Mag\to Dend$ is injective. Since this is also the composite $ComMag \to PreLie\to Dend$ by Lemma \ref{commdiag}, the map $ComMag \to PreLie$ is injective. So, we are done with the second proof.

\vskip0.5cm
Observe that $Prelie\to Dend$ is also injective, it has been proved by M.\ Ronco in \cite{Ronco02}.

\subsection{Splitting of $Dend$} Since there is an injection $ComMag \to Dend$, it is natural to ask the
 same question as in section \ref{factorization}: does there exist an $\Sy$-module $\mathcal{Y}$ such that
 $$Dend= \mathcal{Y}\circ ComMag\ ?$$
If so, the dimensions of $\mathcal{Y}(n)$ are 
$$1, 3, 18, 168, 2130, \ldots\ .$$
More precisely its generating series is $f^{ \mathcal{Y}}(t)= \frac{1}{1-3t-t^3}$.

Since, as an $\Sy$-module, $Dend=Dup$, and since by \cite{JLLgbto} we know that $Dup=As\circ Mag$, it would suffice to find $\YY$ such that $Mag=\YY\circ ComMag$.

\end{document}